\documentclass[reqno]{amsart}
\usepackage[english]{babel}
\usepackage{amssymb,amsmath,hyperref}
\usepackage{amsrefs}
\usepackage[foot]{amsaddr}
\usepackage{mathrsfs}
\usepackage{cleveref, enumerate}

\usepackage{bbold}

\newtheorem{thm}{Theorem}

\newtheorem{cor}[thm]{Corollary}
\newtheorem{defi}[thm]{Definition}
\newtheorem{rem}[thm]{Remark}
\newtheorem{nota}[thm]{Notation}

\newtheorem{princ}[thm]{Principle}

\newtheorem{tempie}[thm]{Template}

\newtheorem*{tempo*}{Template}

\newcommand\be{\begin{equation}}
\newcommand\ee{\end{equation}}

\newbox\gnBoxA
\newdimen\gnCornerHgt
\setbox\gnBoxA=\hbox{$\ulcorner$}
\global\gnCornerHgt=\ht\gnBoxA
\newdimen\gnArgHgt

\def\Godelnum #1{%
	\setbox\gnBoxA=\hbox{$#1$}%
	\gnArgHgt=\ht\gnBoxA%
	\ifnum \gnArgHgt<\gnCornerHgt
		\gnArgHgt=0pt%
	\else
		\advance \gnArgHgt by -\gnCornerHgt%
	\fi
	\raise\gnArgHgt\hbox{$\ulcorner$} \box\gnBoxA %
		\raise\gnArgHgt\hbox{$\urcorner$}}
\def\bdefi{\begin{defi}\rm}
\def\edefi{\end{defi}}
\def\bnota{\begin{nota}\rm}
\def\enota{\end{nota}}
\def\brem{\begin{rem}\rm}
\def\erem{\end{rem}}

\def\STP{\textup{\textsf{STP}}}

\def\H{\textup{\textsf{H}}}
\def\RCA{\textup{\textsf{RCA}}}

\def\RCAo{\textup{\textsf{RCA}}_{0}^{\omega}}
\def\RCAO{\textup{\textsf{RCA}}_{0}^{\Omega}}

\def\ef{\textup{\textsf{ef}}}

\def\WKL{\textup{\textsf{WKL}}}

\def\IVT{\textup{IVT}}
\def\IVT{\textup{\textsf{IVT}}}
\def\UWKL{\textup{\textsf{UWKL}}}

\def\T{\mathcal{T}}

\def\bye{\end{document}}

\def\P{\textup{\textsf{P}}}

\def\R{{\mathbb  R}}

\def\I{{\textsf{\textup{I}}}}

\def\FAN{\textup{\textsf{FAN}}}

\def\MUC{\textup{\textsf{MUC}}}
\def\LLPO{{\textup{\textsf{LLPO}}}}
\def\R{{\mathbb{R}}}
\def\({\textup{(}}
\def\){\textup{)}}

\def\st{\textup{st}}
\def\nst{\textup{nst}}
\def\asa{\leftrightarrow}

\def\di{\rightarrow}

\def\eps{\varepsilon}
\def\epsa{\varepsilon_{\alpha}}

\def\paai{\Pi_{1}^{0}\textup{-\textsf{TRANS}}}

\def\QFAC{\textup{\textsf{QF-AC}}}

\newcommand{\tup}{\underline} 

\def\HIP{\textup{\textsf{HIP}}}
\def\IVT{\textup{\textsf{IVT}}}
\def\forallst{\forall^{\st}}
\def\existsst{\exists^{\st}}
\def\NCR{\textup{\textsf{NCR}}}
\def\intern{\textup{\textsf{int}}}
\def\ZFC{\textup{\textsf{ZFC}}}
\def\IST{\textup{\textsf{IST}}}
\def\NUC{\textup{\textsf{NUC}}}
\def\NSR{\textup{\textsf{NSR}}}
\def\WMX{\textup{\textsf{WMX}}}
\def\FC{\textup{\textsf{FC}}}
\def\TOF{\textup{\textsf{TOF}}}
\def\ECF{\textup{\textsf{ECF}}}
\def\MPC{\textup{\textsf{MPC}}}
\def\EPA{\textup{\textsf{E-PA}}}
\def\NSC{\textup{\textsf{NSC}}}
\def\NSD{\textup{\textsf{NSD}}}
\def\wat{\textup{\textsf{wat}}}

\def\SCF{\textup{\textsf{SCF}}}

\def\MU{\textup{\textsf{MU}}}

\def\HAC{\textup{\textsf{HAC}}}

\def\INT{\textup{\textsf{int}}}

\numberwithin{equation}{section}
\numberwithin{thm}{section}

\usepackage{comment}

\begin{document}
\title{Nonstandard Analysis and Constructivism!}
\author{Sam Sanders}
\address{Munich Center for Mathematical Philosophy, LMU Munich, Germany \& Department of Mathematics, Ghent University}
\email{sasander@me.com}
\keywords{Nonstandard Analysis, constructive mathematics, computational content}
\begin{abstract}
Almost two decades ago, Wattenberg published a paper with the title \emph{Nonstandard Analysis and Constructivism?} in which he speculates on a possible connection between Nonstandard Analysis and constructive mathematics.   We study Wattenberg's work in light of recent research on the aforementioned connection.  On one hand, with only slight modification, some of Wattenberg's theorems in Nonstandard Analysis are seen to yield effective and constructive theorems (not involving Nonstandard Analysis).  On the other hand, we establish the incorrectness of some of Wattenberg's (explicit and implicit) claims regarding the  constructive status of the axioms \emph{Transfer} and \emph{Standard Part} of Nonstandard Analysis.  
\end{abstract}


\maketitle
\thispagestyle{empty}

\section{Introduction}\label{intro}
\noindent
The introduction of Wattenberg's paper \cite{watje} includes the following statement:
\begin{quote}
This is a speculative paper. For some time the author has been struck by an apparent affinity between two rather unlikely areas of mathematics - nonstandard analysis and constructivism. [\dots] 
The purpose of this paper is to investigate these ideas by examining several examples. (\cite{watje}*{p.\ 303})
\end{quote}
In a nutshell, the aim of this paper is to study Wattenberg's results in light of recent results on the computational content of Nonstandard Analysis as in \cites{samGH, samzoo, samzooII, sambon}.

\medskip

First of all, \emph{similar} observations concerning the constructive content of Nonstandard Analysis have been made before, e.g.\ as follows:
\begin{quote}
{It has often been held that nonstandard analysis is highly non-constructive, thus somewhat suspect, depending as it does upon the ultrapower construction to produce a model \textup{[\dots]} On the other hand, nonstandard \emph{praxis} is remarkably constructive; having the extended number set we can proceed with explicit calculations.} (Emphasis in original: \cite{NORSNSA}*{p.\ 31})
\end{quote}
Like-minded statements may be found in \cites{rossenaap, rosse,kifar, kieken, Oss3, Oss2, sc, nsawork2, venice,fath, jep}.  The reader may interpret the word \emph{constructive} as the mainstream/classical notion `effective', or as the foundational notion from Bishop's \emph{Constructive Analysis} (\cite{bish1}).  
As will become clear, both cases will be treated below (and separated carefully).      

\medskip

However, Wattenberg goes further than most of the aforementioned authors by making the following important observation.
\begin{quote}
Despite an essential nonconstructive kernel, many nonstandard arguments are constructive until the final step, a step that frequently involves the standard part map. (\cite{watje}*{p.\ 303})
\end{quote}
This observation is similar to Osswald's \emph{local constructivity}.  In particular, Osswald has qualified the observation from the above quotes as \emph{Nonstandard Analysis is locally constructive}, to be understood as the fact that the mathematics performed in the nonstandard world is highly constructive\footnote{The mathematics performed in the nonstandard world usually amounts to merely manipulating sums and products of nonstandard length.}.  By contrast, the nonstandard axioms (\emph{Transfer} and \emph{Standard Part}) needed to `jump between' the nonstandard world and usual mathematics, are \emph{highly non-constructive} in general.  Osswald discusses local constructivity in \cite{nsawork2}*{\S7}, \cite{Oss3}*{\S1-2}, or \cite{Oss2}*{\S17.5}.  

\medskip

The results in \cites{samGH, samzoo, samzooII, sambon} vindicate the Wattenberg and Osswald view in that computational content is extracted from theorems of `pure' Nonstandard Analysis, i.e.\ formulated solely with the \emph{nonstandard} definitions (of continuity, Riemann integration, compactness, et cetera) rather than the `$\eps$-$\delta$' definitions.  
With this choice, one only works in the nonstandard universe, avoiding the non-constructive step from and to the standard/usual universe (requiring \emph{Transfer} and \emph{Standard Part}). 

\medskip

In this paper, we show that Wattenberg's results from \cite{watje} yield effective and constructive results with only slight modification.  However, we also establish the incorrectness of Wattenberg's claims regarding the constructive status of the nonstandard axioms \emph{Transfer} and \emph{Standard Part}.  
In contrast to Wattenberg, we shall work in Nelson's axiomatic approach to Nonstandard Analysis (See Section \ref{P}), but this change of framework will have no real impact on our results or Wattenberg's.

\section{Internal set theory and its fragments}\label{P}
In this section, we discuss Nelson's \emph{internal set theory}, first introduced in \cite{wownelly}, and its fragments $\P$ and $\H$ from \cite{brie}.  
The latter fragments are essential to our enterprise, especially Theorem~\ref{consresultcor} below.  
\subsection{Internal set theory 101}\label{IIST}
In Nelson's \emph{syntactic} approach to Nonstandard Analysis (\cite{wownelly}), as opposed to Robinson's semantic one (\cite{robinson1}), a new predicate `st($x$)', read as `$x$ is standard' is added to the language of \textsf{ZFC}, the usual foundation of mathematics.  
The notations $(\forall^{\st}x)$ and $(\exists^{\st}y)$ are short for $(\forall x)(\st(x)\di \dots)$ and $(\exists y)(\st(y)\wedge \dots)$.  A formula is called \emph{internal} if it does not involve `st', and \emph{external} otherwise.   
The three external axioms \emph{Idealisation}, \emph{Standard Part}, and \emph{Transfer} govern the new predicate `st'.  These axioms are respectively defined\footnote{The superscript `fin' in \textsf{(I)} means that $x$ is finite, i.e.\ its number of elements are bounded by a natural number.} as:  
\begin{enumerate}
\item[\textsf{(I)}] $(\forall^{\st~\textup{fin}}x)(\exists y)(\forall z\in x)\varphi(z,y)\di (\exists y)(\forall^{\st}x)\varphi(x,y)$, for internal $\varphi$.
\item[\textsf{(S)}] $(\forall^{\st} x)(\exists^{\st}y)(\forall^{\st}z)\big((z\in x\wedge \varphi(z))\asa z\in y\big)$, for any $\varphi$.
\item[\textsf{(T)}] $(\forall^{\st}t)\big[(\forall^{\st}x)\varphi(x, t)\di (\forall x)\varphi(x, t)\big]$, where $\varphi(x,t)$ is internal, and only has free variables $t, x$.  
\end{enumerate}
The system \textsf{IST} is (the internal system) \textsf{ZFC} extended with the aforementioned external axioms;  
The former is a conservative extension of \textsf{ZFC} for the internal language, as proved in \cite{wownelly}.    
Clearly, the extension from $\ZFC$ to $\IST$ can be done for other systems, and we are interested in the formalisations of (classical and intuitionistic) arithmetic, namely \emph{Peano} and \emph{Heyting} arithmetic.  In this regard, the systems $\H$ and $\P$ from \cite{brie}, also sketched in the next sections, are nonstandard extensions of the (internal) logical systems \textsf{E-HA}$^{\omega}$ and $\textsf{E-PA}^{\omega}$, respectively \emph{Heyting and Peano arithmetic in all finite types and the axiom of extensionality}.       
We refer to \cite{kohlenbach3}*{\S3.3} for the exact definitions of the (mainstream in mathematical logic) systems \textsf{E-HA}$^{\omega}$ and $\textsf{E-PA}^{\omega}$.  \
%
%
\subsection{The classical system $\P$}\label{PIPI}
In this section, we introduce the system $\P$, a conservative extension of Peano arithmetic with fragments of Nelson's $\IST$.  

\medskip

To this end, recall that \textsf{E-PA}$^{\omega*}$ and $\textsf{E-HA}^{\omega*}$ are the definitional extensions of \textsf{E-PA}$^{\omega}$ and $\textsf{E-HA}^{\omega}$ with types for finite sequences, as in \cite{brie}*{\S2}.  For the former systems, we require some notation.  
\begin{nota}[Finite sequences]\label{skim}\rm
The systems $\textsf{E-PA}^{\omega*}$ and $\textsf{E-HA}^{\omega*}$ have a dedicated type for `finite sequences of objects of type $\rho$', namely $\rho^{*}$.  Since the usual coding of pairs of numbers goes through in both, we shall not always distinguish between $0$ and $0^{*}$.  
Similarly, we do not always distinguish between `$s^{\rho}$' and `$\langle s^{\rho}\rangle$', where the former is `the object $s$ of type $\rho$', and the latter is `the sequence of type $\rho^{*}$ with only element $s^{\rho}$'.  The empty sequence for the type $\rho^{*}$ is denoted by `$\langle \rangle_{\rho}$', usually with the typing omitted.  Furthermore, we denote by `$|s|=n$' the length of the finite sequence $s^{\rho^{*}}=\langle s_{0}^{\rho},s_{1}^{\rho},\dots,s_{n-1}^{\rho}\rangle$, where $|\langle\rangle|=0$, i.e.\ the empty sequence has length zero.
For sequences $s^{\rho^{*}}, t^{\rho^{*}}$, we denote by `$s*t$' the concatenation of $s$ and $t$, i.e.\ $(s*t)(i)=s(i)$ for $i<|s|$ and $(s*t)(j)=t(|s|-j)$ for $|s|\leq j< |s|+|t|$. For a sequence $s^{\rho^{*}}$, we define $\overline{s}N:=\langle s(0), s(1), \dots,  s(N)\rangle $ for $N^{0}<|s|$.  
For a sequence $\alpha^{0\di \rho}$, we also write $\overline{\alpha}N=\langle \alpha(0), \alpha(1),\dots, \alpha(N)\rangle$ for \emph{any} $N^{0}$.  By way of shorthand, $q^{\rho}\in Q^{\rho^{*}}$ abbreviates $(\exists i<|Q|)(Q(i)=_{\rho}q)$.  Finally, we shall use $\underline{x}, \underline{y},\underline{t}, \dots$ as short for tuples $x_{0}^{\sigma_{0}}, \dots x_{k}^{\sigma_{k}}$ of possibly different type $\sigma_{i}$.          
\end{nota}

We can now introduce $\textsf{E-PA}_{\st}^{\omega*}$.  
We use the same definition as \cite{brie}*{Def.~6.1}, where \textsf{E-PA}$^{\omega*}$ is the definitional extension of \textsf{E-PA}$^{\omega}$ with types for finite sequences from \cite{brie}*{\S2}.  
The set $\T^{*}$ is the collection of all the terms in the language of $\textsf{E-PA}^{\omega*}$.    
\bdefi\label{debs}
The system $ \textsf{E-PA}^{\omega*}_{\st} $ is defined as $ \textsf{E-PA}^{\omega{*}} + \T^{*}_{\st} + \textsf{IA}^{\st}$, where $\T^{*}_{\st}$
consists of the following \emph{basic} axiom schemas.
\begin{enumerate}
\item The schema\footnote{The language of $\textsf{E-PA}_{\st}^{\omega*}$ contains a symbol $\st_{\sigma}$ for each finite type $\sigma$, but the subscript is essentially always omitted.  Hence $\T^{*}_{\st}$ is an \emph{axiom schema} and not an axiom.\label{omit}} $\st(x)\wedge x=y\di\st(y)$,
\item The schema providing for each closed\footnote{A term is called \emph{closed} in \cite{brie} (and in this paper) if all variables are bound via lambda abstraction.  Thus, if $\underline{x}, \underline{y}$ are the only variables occurring in the term $t$, the term $(\lambda \underline{x})(\lambda\underline{y})t(\underline{x}, \underline{y})$ is closed while $(\lambda \underline{x})t(\underline{x}, \underline{y})$ is not.  The second axiom in Definition \ref{debs} thus expresses that $\st_{\tau}\big((\lambda \underline{x})(\lambda\underline{y})t(\underline{x}, \underline{y})\big)$ if $(\lambda \underline{x})(\lambda\underline{y})t(\underline{x}, \underline{y})$ is of type $\tau$.  We usually omit lambda abstraction for brevity.\label{kootsie}} term $t\in \T^{*}$ the axiom $\st(t)$.
\item The schema $\st(f)\wedge \st(x)\di \st(f(x))$.
\end{enumerate}
The external induction axiom \textsf{IA}$^{\st}$ is as follows.  
\be\tag{\textsf{IA}$^{\st}$}
\Phi(0)\wedge(\forall^{\st}n^{0})(\Phi(n) \di\Phi(n+1))\di(\forall^{\st}n^{0})\Phi(n).     
\ee
\edefi
Secondly, we introduce some essential fragments of $\IST$ studied in \cite{brie}.  
\bdefi[External axioms of $\P$]~
\begin{enumerate}
\item$\HAC_{\INT}$: For any internal formula $\varphi$, we have
\be\label{HACINT}
(\forall^{\st}x^{\rho})(\exists^{\st}y^{\tau})\varphi(x, y)\di \big(\exists^{\st}F^{\rho\di \tau^{*}}\big)(\forall^{\st}x^{\rho})(\exists y^{\tau}\in F(x))\varphi(x,y),
\ee
\item $\textsf{I}$: For any internal formula $\varphi$, we have
\[
(\forall^{\st} x^{\sigma^{*}})(\exists y^{\tau} )(\forall z^{\sigma}\in x)\varphi(z,y)\di (\exists y^{\tau})(\forall^{\st} x^{\sigma})\varphi(x,y), 
\]
\item The system $\P$ is $\textsf{E-PA}_{\st}^{\omega*}+\textsf{I}+\HAC_{\INT}$.
\end{enumerate}
\end{defi}
Note that \textsf{I} and $\HAC_{\INT}$ are fragments of Nelson's axioms \emph{Idealisation} and \emph{Standard part}.  
By definition, $F$ in \eqref{HACINT} only provides a \emph{finite sequence} of witnesses to $(\exists^{\st}y)$, explaining its name \emph{Herbrandized Axiom of Choice}.   

\medskip

The system $\P$ is connected to $\textsf{E-PA}^{\omega}$ by Theorem \ref{consresultcor}.    
%
The latter (which is not present in \cite{brie}) expresses that we may obtain effective results as in \eqref{effewachten} from any theorem of Nonstandard Analysis which has the same form as in \eqref{bog}.   
\begin{thm}\label{consresultcor}
If $\Delta_{\intern}$ is a collection of internal formulas and $\psi$ is internal, and
\be\label{bog}
\P + \Delta_{\intern} \vdash (\forall^{\st}\underline{x})(\exists^{\st}\underline{y})\psi(\underline{x},\underline{y}, \underline{a}), 
\ee
then one can extract from the proof a sequence of closed$^{\ref{kootsie}}$ terms $t$ in $\mathcal{T}^{*}$ such that
\be\label{effewachten}
\textup{\textsf{E-PA}}^{\omega*} + \Delta_{\intern} \vdash (\forall \underline{x})(\exists \underline{y}\in t(\underline{x}))\psi(\underline{x},\underline{y},\underline{a}).
\ee
\end{thm}
\begin{proof}
See e.g.\ \cite{samzoo}*{\S2} or \cite{sambon}*{\S2}.
\end{proof}
For the rest of this paper, the notion `normal form' shall refer to a formula as in \eqref{bog}, i.e.\ of the form $(\forall^{\st}x)(\exists^{\st}y)\varphi(x,y)$ for $\varphi$ internal.  
It is shown in \cites{sambon, samzoo, samzooII} that the scope of Theorem \ref{consresultcor} includes the `Big Five' systems of Reverse Mathematics and the associated `zoo' from \cite{damirzoo}.  
\medskip

Finally, the previous theorems do not really depend on the presence of full Peano arithmetic.  
We shall study the following subsystems.   
\bdefi[Weaker Systems]
\begin{enumerate}
\item Let \textsf{E-PRA}$^{\omega}$ be the system defined in \cite{kohlenbach2}*{\S2} and let \textsf{E-PRA}$^{\omega*}$ 
be its definitional extension with types for finite sequences as in \cite{brie}*{\S2}. 
\item $(\QFAC^{\rho, \tau})$ For every quantifier-free internal formula $\varphi(x,y)$, we have
\be\label{keuze}
(\forall x^{\rho})(\exists y^{\tau})\varphi(x,y) \di (\exists F^{\rho\di \tau})(\forall x^{\rho})\varphi(x,F(x))
\ee
\item The system $\RCAo$ is $\textsf{E-PRA}^{\omega}+\QFAC^{1,0}$.  
\end{enumerate}
\edefi
The system $\RCAo$ is the `base theory of higher-order Reverse Mathematics' as introduced in \cite{kohlenbach2}*{\S2}.  
We permit ourselves a slight abuse of notation by also referring to the system $\textsf{E-PRA}^{\omega*}+\QFAC^{1,0}$ as $\RCAo$.
\begin{cor}\label{consresultcor2}
The previous theorem and corollary go through for $\P$ and $\textsf{\textup{E-PA}}^{\omega*}$ replaced by $\P_{0}\equiv \textsf{\textup{E-PRA}}^{\omega*}+\T_{\st}^{*} +\HAC_{\INT} +\textsf{\textup{I}}+\QFAC^{1,0}$ and $\RCAo$.  
\end{cor}
\begin{proof}
The proof of \cite{brie}*{Theorem 7.7} goes through for any fragment of \textsf{E-PA}$^{\omega{*}}$ which includes \textsf{EFA}, sometimes also called $\textsf{I}\Delta_{0}+\textsf{EXP}$.  
In particular, the exponential function is (all what is) required to `easily' manipulate finite sequences.    
\end{proof}
We now discuss the \emph{Standard Part} principle $\Omega$\textsf{-CA}, a very practical consequence of the axiom $\HAC_{\INT}$.  
Intuitively speaking, $\Omega$\textsf{-CA} expresses that we can obtain 
the standard part (in casu $G$) of \emph{$\Omega$-invariant} nonstandard objects (in casu $F(\cdot,M)$).   
Note that we write `$N\in \Omega$' as short for $\neg\st(N^{0})$.
\bdefi[$\Omega$-invariance]\label{homega} Let $F^{(\sigma\times  0)\di 0}$ be standard and fix $M^{0}\in \Omega$.  
Then $F(\cdot,M)$ is {\bf $\Omega$-invariant} if   
\be\label{homegainv}
(\forall^{\st} x^{\sigma})(\forall N^{0}\in \Omega)\big[F(x ,M)=_{0}F(x,N) \big].  
\ee
\edefi
\begin{princ}[$\Omega$\textsf{-CA}]\rm Let $F^{(\sigma\times 0)\di 0}$ be standard and fix $M^{0}\in \Omega$.
For every $\Omega$-invariant $F(\cdot,M)$, there is a standard $G^{\sigma\di 0}$ such that
\be\label{homegaca}
(\forall^{\st} x^{\sigma})(\forall N^{0}\in \Omega)\big[G(x)=_{0}F(x,N) \big].  
\ee
\end{princ}
The axiom $\Omega$\textsf{-CA} provides the standard part of a nonstandard object, if the latter is \emph{independent of the choice of nonstandard number} used in its definition.  
The following theorem is not new, but highly instructive in light of Remark \ref{simply}.    
\begin{thm}\label{drifh}
The system $\P_{0}$ proves $\Omega\textup{\textsf{-CA}}$.  
\end{thm}
\begin{proof}
Let $F(\cdot,M^{0})$ be $\Omega$-invariant, i.e.\ we have 
\be\label{dorkillllll}
(\forall^{\st} x^{\sigma})(\forall N^{0},M^{0}\in \Omega)\big[F(x ,M)=_{0}F(x,N) \big].  
\ee
By underspill (See Theorem \ref{doppi} below), we immediately obtain that 
\[
(\forall^{\st} x^{\sigma})(\exists^{\st} k^{0})(\forall N^{0},M^{0}\geq k)\big[F(x ,M)=_{0}F(x,N) \big].  
\]
Now apply $\HAC_{\INT}$ to obtain standard $\Phi^{\sigma\di 0^{*}}$ such that
\[
(\forall^{\st} x^{\sigma})(\exists k^{0}\in \Phi(x))(\forall N^{0},M^{0}\geq k)\big[F(x ,M)=_{0}F(x,N) \big].
\]
Next, define $\Psi(x):= \max_{i<|\Phi(x)|}\Phi(x)(i)$ and note that 
\[
(\forall^{\st} x^{\sigma})(\forall N^{0},M^{0}\geq \Psi(x))\big[F(x ,M)=_{0}F(x,N) \big].
\]
Finally, put $G(x):=F(x,\Psi(x))$ and note that $\Omega$\textsf{-CA} follows.
\end{proof}
Finally, we consider the following remark on how $\HAC_{\INT}$ and $\textsf{I}$ are used.  
\begin{rem}[Using $\HAC_{\INT}$ and $\textsf{I}$]\label{simply}\rm
By definition, $\HAC_{\INT}$ produces a functional $F^{\sigma\di \tau^{*}}$ which outputs a \emph{finite sequence} of witnesses.  
However, $\HAC_{\INT}$ provides an actual \emph{witnessing functional} assuming (i) $\tau=0$ in $\HAC_{\INT}$ and (ii) the formula $\varphi$ from $\HAC_{\INT}$ is `sufficiently monotone' as in: 
$(\forall^{\st} x^{\sigma},n^{0},m^{0})\big([n\leq_{0}m \wedge\varphi(x,n)] \di \varphi(x,m)\big)$.    
Indeed, in this case one simply defines $G^{\sigma\di 0}$ by $G(x^{\sigma}):=\max_{i<|F(x)|}F(x)(i)$ which satisfies $(\forall^{\st}x^{\sigma})\varphi(x, G(x))$, as was done in the proof of Theorem \ref{drifh}.  
To save space in proofs, we will sometimes skip the (obvious) step involving the maximum of finite sequences, when applying $\HAC_{\INT}$.  
We assume the same convention for terms obtained from Theorem \ref{consresultcor}, and applications of the contraposition of \emph{Idealisation} \textsf{I}.  
\end{rem}

\subsection{The constructive system $\H$}
In this section, we define the system $\H$, the constructive counterpart of $\P$. 
The system $\textsf{H}$ was first introduced in \cite{brie}*{\S5.2}, and constitutes a conservative extension of Heyting arithmetic $\textup{\textsf{E-HA}}^{\omega} $ by \cite{brie}*{Cor.\ 5.6}.

\medskip

Similar to Definition \ref{debs}, we define $ \textsf{E-HA}^{\omega*}_{\st} $ as $ \textsf{E-HA}^{\omega{*}} + \T^{*}_{\st} + \textsf{IA}^{\st}$, where $\textsf{E-HA}^{\omega*}$ is essentially just $\textsf{E-PA}^{\omega*}$ without the law of excluded middle.  
Furthermore, define
\[
\H\equiv \textup{\textsf{E-HA}}^{\omega*}_{\st}+\HAC + {\I}+\NCR+\textsf{HIP}_{\forall^{\st}}+\textsf{HGMP}^{\st},
\]
where $\HAC$ is $\HAC_{\INT}$ without any restriction on the formula, and where the remaining axioms are defined in the following definition.
\bdefi[Three axioms of $\H$]\label{flah}~
\begin{enumerate}\rm
\item $\textsf{HIP}_{\forall^{\st}}$
\[
[(\forall^{\st}x)\phi(x)\di (\exists^{\st}y)\Psi(y)]\di (\exists^{\st}y')[(\forall^{\st}x)\phi(x)\di (\exists y\in y')\Psi(y)],
\]
where $\Psi(y)$ is any formula and $\phi(x)$ is an internal formula of \textsf{E-HA}$^{\omega*}$. 
\item $\textsf{HGMP}^{\st}$
\[
[(\forall^{\st}x)\phi(x)\di \psi] \di (\exists^{\st}x')[(\forall x\in x')\phi(x)\di \psi] 
\]
where $\phi(x)$ and $\psi$ are internal formulas in the language of \textsf{E-HA}$^{\omega*}$.
\item \textsf{NCR}
\[
(\forall y^{\tau})(\exists^{\st} x^{\rho} )\Phi(x, y) \di (\exists^{\st} x^{\rho^{*}})(\forall y^{\tau})(\exists x'\in x )\Phi(x', y),
\]
where $\Phi$ is any formula of \textsf{E-HA}$^{\omega*}$
\end{enumerate}
\edefi
Intuitively speaking, the first two axioms of Definition \ref{flah} allow us to perform a number of \emph{non-constructive operations} (namely \emph{Markov's principle} and \emph{independence of premises}) 
on the standard objects of the system $\H$, provided we introduce a `Herbrandisation' as in the consequent of $\HAC$, i.e.\ a finite list of possible witnesses rather than one single witness. 
Furthermore, while $\H$ includes \emph{Idealisation} \textsf{I}, one often uses the latter's \emph{classical contraposition}, explaining why \textsf{NCR} is useful (and even essential) in the context of intuitionistic logic.  

\medskip

Surprisingly, the axioms from Definition \ref{flah} are exactly what is needed to convert nonstandard definitions (of continuity, integrability, convergence, et cetera) into the normal form $(\forall^{\st}x)(\exists^{\st}y)\varphi(x, y)$ for internal $\varphi$, as done in \cite{sambon}*{\S3.1}.
The latter normal form plays an equally important role in the constructive case as in the classical case by the following theorem.  
\begin{thm}\label{consresult2}
If $\Delta_{\intern}$ is a collection of internal formulas, $\varphi$ is internal, and
\be\label{antecedn3}
\textup{\textsf{H}} + \Delta_{\intern} \vdash \forallst \tup x \, \existsst \tup y \, \varphi(\tup x, \tup y, \tup a), 
\ee
then one can extract from the proof a sequence of closed terms $t$ in $\mathcal{T}^{*}$ such that
\be\label{consequalty3}
\textup{\textsf{E-HA}}^{\omega*} + \Delta_{\intern} \vdash\  \forall \tup x \, \exists \tup y\in \tup t(\tup x)\ \varphi(\tup x,\tup y, \tup a).
\ee
\end{thm}
\begin{proof}
Immediate by \cite{brie}*{Theorem 5.9}.  Note that in the latter, just like in the proof of Corollary \ref{consresultcor}, $\forallst \tup x \, \existsst \tup y \, \varphi(\tup x, \tup y, \tup a) $ is proved to be `invariant' under a suitable syntactic translation.  
\end{proof}
Finally, we point out some very useful principles to which we have access.  
\begin{thm}\label{doppi}
The systems $\P$, $\P_{0}$, and $\H$ prove \emph{overspill} and \emph{underspill}, i.e.\
\[
(\forall^{\st}x^{\rho})\varphi(x)\di (\exists y^{\rho})\big[\neg\st(y)\wedge \varphi(y)  \big] \textup{ and }(\forall x^{\rho})\big[\neg\st(x)\di \varphi(x)]\di (\exists^{\st} y^{\rho})\varphi(y),
\]
for any internal formula $\varphi$.
\end{thm}
\begin{proof}
Immediate by \cite{brie}*{Prop.\ 3.3 and 5.11}.  
\end{proof}
We will apply underspill most frequently as follows:  From $(\forall M\in \Omega)\psi(M)$ for internal $\psi$, we conclude that $(\forall K^{0})\big[\neg\st(K)\di (\forall M\geq K)\psi(M) \big]$.  Applying underspill for $\varphi(K)\equiv (\forall M\geq K)\psi(M)$, we obtain $(\exists^{\st} K^{0}) (\forall M\geq K)\psi(M)$.   

\medskip

In conclusion, we have introduced the systems $\H$, $\P$, which are conservative extensions of Peano and Heyting arithmetic with fragments of Nelson's internal set theory.  
We have observed that central to the conservation result in Theorem~\ref{consresultcor} is the normal form $(\forall^{\st}x)(\exists^{\st}y)\varphi(x, y)$ for internal $\varphi$.  

\subsection{Notations}
In this section, we introduce notations relating to $\H$ and $\P$.  

\medskip

First of all, we mostly use the same notations as in \cite{brie}.  
\begin{rem}[Notations]\label{notawin}\rm
We write $(\forall^{\st}x^{\tau})\Phi(x^{\tau})$ and $(\exists^{\st}x^{\sigma})\Psi(x^{\sigma})$ as short for 
$(\forall x^{\tau})\big[\st(x^{\tau})\di \Phi(x^{\tau})\big]$ and $(\exists x^{\sigma})\big[\st(x^{\sigma})\wedge \Psi(x^{\sigma})\big]$.     
We also write $(\forall x^{0}\in \Omega)\Phi(x^{0})$ and $(\exists x^{0}\in \Omega)\Psi(x^{0})$ as short for 
$(\forall x^{0})\big[\neg\st(x^{0})\di \Phi(x^{0})\big]$ and $(\exists x^{0})\big[\neg\st(x^{0})\wedge \Psi(x^{0})\big]$.  Furthermore, $\neg\st(x^{0})$, is abbreviated by `$x^{0}\in \Omega$'.  
A formula $A$ is `internal' if it does not involve $\st$, and `external' otherwise.  The formula $A^{\st}$ is defined from $A$ by appending `st' to all quantifiers (except bounded number quantifiers).    
\end{rem}
Secondly, we will use the usual notations for natural, rational and real numbers and functions as introduced in \cite{kohlenbach2}*{p.\ 288-289}. 
We only list the definition of real number and related notions in $\P$ and related systems.
\begin{defi}[Real numbers and related notions in $\P$]\label{keepintireal}\rm~
\begin{enumerate}
\item A (standard) real number $x$ is a (standard) fast-converging Cauchy sequence $q_{(\cdot)}^{1}$, i.e.\ $(\forall n^{0}, i^{0})(|q_{n}-q_{n+i})|<_{0} \frac{1}{2^{n}})$.  
We use Kohlenbach's `hat function' from \cite{kohlenbach2}*{p.\ 289} to guarantee that every sequence $f^{1}$ is a real.  
\item We write $[x](k):=q_{k}$ for the $k$-th approximation of a real $x^{1}=(q^{1}_{(\cdot)})$.    
\item Two reals $x, y$ represented by $q_{(\cdot)}$ and $r_{(\cdot)}$ are \emph{equal}, denoted `$x=_{\R}y$', if $(\forall n)(|q_{n}-r_{n}|\leq \frac{1}{2^{n-1}})$. Inequality `$<_{\R}$' is defined similarly.         
\item We  write `$x\approx y$' if $(\forall^{\st} n)(|q_{n}-r_{n}|\leq \frac{1}{2^{n-1}})$ and $x\gg y$ if $x>y\wedge x\not\approx y$.  
\item Functions $F:\R\di \R$ mapping reals to reals are represented by functionals $\Phi^{1\di 1}$ mapping equal reals to equal reals, i.e. 
\be\tag{\textsf{RE}}\label{furg}
(\forall x, y)(x=_{\R}y\di \Phi(x)=_{\R}\Phi(y)).
\ee
 \item For a space $X$ with metric $|\cdot|_{X}:X\di \R$, we write `$x\approx y$' for `$|x-y|_{X}\approx 0$'.    
\item Sets of objects of type $\rho$ are denoted $X^{\rho\di 0}, Y^{\rho\di }, Z^{\rho\di 0}, \dots$ and are given by their characteristic functions $f^{\rho\di 0}_{X}$, i.e.\ $(\forall x^{\rho})[x\in X\asa f_{X}(x)=_{0}1]$, where $f_{X}^{\rho\di 0}$ is assumed to output zero or one.  
\end{enumerate}
\end{defi}
%
%
Thirdly, we use the usual extensional notion of equality. 
\begin{rem}[Equality]\label{equ}\rm
All the above systems include equality between natural numbers `$=_{0}$' as a primitive.  Equality `$=_{\tau}$' for type $\tau$-objects $x,y$ is defined as:
\be\label{aparth}
[x=_{\tau}y] \equiv (\forall z_{1}^{\tau_{1}}\dots z_{k}^{\tau_{k}})[xz_{1}\dots z_{k}=_{0}yz_{1}\dots z_{k}]
\ee
if the type $\tau$ is composed as $\tau\equiv(\tau_{1}\di \dots\di \tau_{k}\di 0)$.  Inequality `$\leq_{\tau}$' is then just \eqref{aparth} with `$=_{0}$' replaced by `$\leq_{0}$'.  
In the spirit of Nonstandard Analysis, we define `approximate equality $\approx_{\tau}$' as follows:
\be\label{aparth2}
[x\approx_{\tau}y] \equiv (\forall^{\st} z_{1}^{\tau_{1}}\dots z_{k}^{\tau_{k}})[xz_{1}\dots z_{k}=_{0}yz_{1}\dots z_{k}]
\ee
with the type $\tau$ as above.  
All the above systems include the \emph{axiom of extensionality}:
\be\label{EXT}\tag{\textsf{E}}  
(\forall  x^{\rho},y^{\rho},\varphi^{\rho\di \tau}) \big[x=_{\rho} y \di \varphi(x)=_{\tau}\varphi(y)   \big].
\ee
However, as noted in \cite{brie}*{p.\ 1973}, the so-called axiom of \emph{standard} extensionality \eqref{EXT}$^{\st}$ is problematic and cannot be included in our systems.  
We use $\eqref{EXT}_{n+2}$ to refer to $\eqref{EXT}$ restricted to type $n+2$ functionals $\varphi$.  
\end{rem}

\subsection{Preliminaries}\label{preint}
In this section, we introduce the usual defintions of continuity, as well as some fragments of \emph{Transfer} and \emph{Standard Part}, their normal forms, and the functionals arising from term extraction as in Theorem~\ref{consresultcor}.  
\subsubsection{Continuity, nonstandard and otherwise}
\bdefi[Continuity]\label{Kont}
A function $f$ is \emph{continuous} on $[0,1]$ if
\be\label{soareyou1}\textstyle
(\forall k^{0})(\forall x\in [0,1])(\exists N^{0})\big[(\forall y\in [0,1])(|x-y|<_{\R}\frac{1}{N}\di |f(x)-f(y)|<_{\R}\frac{1}{k})\big].
\ee
A function $f$ is \emph{nonstandard continuous} on $[0,1]$ if
\be\label{soareyou3}
(\forall^{\st}x\in [0,1])(\forall y\in [0,1])[x\approx y \di f(x)\approx f(y)].
\ee
A function $f$ is \emph{uniformly continuous} on $[0,1]$ if
\be\label{soareyou2}\textstyle
(\forall k^{0})(\exists N^{0})\big[(\forall x, y\in [0,1])(|x-y|<_{\R}\frac{1}{N}\di |f(x)-f(y)|<_{\R}\frac{1}{k})\big].
\ee
A function $f$ is \emph{nonstandard uniformly continuous} on $[0,1]$ if
\be\label{soareyou4}
(\forall x, y\in [0,1])[x\approx y \di f(x)\approx f(y)].
\ee
\edefi
\bdefi[Modulus of continuity]\label{mikeh}~
\begin{enumerate}
\item A function $g$ which provides $N$ as in \eqref{soareyou1} (resp.\ \eqref{soareyou2}) for any $k^{0}$ and $ x^{1}\in [0,1]$ (resp.\ for $k^{0}$) is a \emph{modulus} of pointwise (resp.\ uniform) continuity.  
\item We write $f\in C([0,1])$ (resp.\ $f\in C^{\st}([0,1])$) for \eqref{soareyou1} (resp.\ \eqref{soareyou1}$^{\st}$).  
\item The principle $\NSC_{1}$ (resp.\ $\NSC_{2}$) is the statement that every standard $f\in C([0,1])$ (resp.\ standard $f\in C^{\st}([0,1])$) is also nonstandard continuous.  
\item The principle $\MPC(\Xi)$ is the statement that $\Xi(f)$ is a modulus of pointwise continuity for every $f\in C([0,1])$.  
\end{enumerate}
\edefi
As will become clear in Section \ref{good}, there is an intimate connection between nonstandard continuity \eqref{soareyou3} and a `modulus-of-continuity functional' $\Xi$ as in $\MPC(\Xi)$.  
\subsubsection{Transfer and comprehension}
We require two equivalent (See \cite{kohlenbach2}*{Prop.\ 3.9}) versions of arithmetical comprehension: 
\be\label{mu}\tag{$\mu^{2}$}
(\exists \mu^{2})\big[(\forall f^{1})( (\exists n)f(n)=0 \di f(\mu(f))=0)    \big],
\ee
\be\label{mukio}\tag{$\exists^{2}$}
(\exists \varphi^{2})\big[(\forall f^{1})( (\exists n)f(n)=0 \asa \varphi(f)=0    \big],
\ee
and also the restriction of Nelson's axiom \emph{Transfer} as follows:
\be\tag{$\paai$}
(\forall^{\st}f^{1})\big[(\forall^{\st}n^{0})f(n)\ne0\di (\forall m)f(m)\ne0\big].
\ee
Denote by $\textsf{MU}(\mu)$ the formula in square brackets in \eqref{mu}.  By the following theorems, $\paai$ is fundamentally non-constructive and closely related to arithmetical comprehension as given by $(\mu^{2})$.  
\begin{thm}\label{klabaka}
The system $\P$ proves $(\exists^{\st}\mu^{2})\MU(\mu)\di \paai \di (\mu^{2})^{\st}$.  
\end{thm}
\begin{proof}
The first implication is immediate as standard functionals produce standard output on standard input by the third basic axiom in Definition \ref{debs}.  In particular, $\mu(f)$ is standard for standard $f^{1}$ if $\mu^{2}$ is standard.  For the second implication, note that $\paai$ implies by contraposition that:
\be\label{toggg}
(\forall^{\st}f^{1})(\exists^{\st}m)\big[(\exists n^{0})f(n)=0\di (\exists i\leq m)f(i)=0\big].
\ee
Applying $\HAC_{\INT}$ while bearing in mind Remark \ref{simply}, we obtain standard $\Phi$ s.t.\ 
\be\label{geluklkt}
(\forall^{\st}f^{1})\big[(\exists n^{0})f(n)=0\di (\exists i\leq \Phi(f))f(i)=0\big], 
\ee
which immediately yields $(\mu^{2})^{\st}$, and we are done. 
\end{proof}

\begin{thm}\label{idare2}
From the proof that $\P_{0}\vdash[ \NSC_{1}\di \paai]$, a term $t$ can be extracted such that $\RCAo\vdash (\forall \Xi^{3})\big(\MPC(\Xi)\di \MU(t(\Xi)))$
\end{thm}
\begin{proof}
See the proof of Corollary \ref{idare}.  
\end{proof}
By the previous theorem, term extraction as in Theorem \ref{consresultcor} converts $\paai$ into $(\mu^{2})$, i.e.\ the former fragment of \emph{Transfer} is rather non-constructive.  

\medskip

Finally, we should point out that it is possible to obtain (some) computational information from (non-constructive) proofs in classical mathematics, even involving comprehension.  This is the domain of \emph{proof mining}, to which Kohlenbach's monograph \cite{kohlenbach3} provides an excellent introduction.  Thus, our use of `$X$ is non-constructive' should be interpreted as the observation that $X$ is rejected in (parts of) constructive mathematics, while $X$ may have implicit constructive content, which can be brought out using proof mining.  In fact, we will use a technique from proof theory to obtain computational content from \emph{Standard Part} in Section~\ref{SSTP}.    

%
\subsubsection{Standard Part and related functionals}\label{firstkol}
We shall make use of the following fragments of the \emph{Standard Part} axiom:
\be\label{STP}\tag{\textsf{STP}}
(\forall \alpha^{1}\leq_{1}1)(\exists^{\st} \beta^{1}\leq_{1}1)(\alpha\approx_{1} \beta),
\ee
\be\label{STPR}\tag{$\textsf{\textup{STP}}_{\R}$}
(\forall x\in [0,1])(\exists^{\st}y\in [0,1])(x\approx y).  
\ee
These fragments are both equivalent to the following by Theorem \ref{klak} below:
\begin{align}\label{fanns}
(\forall T^{1}\leq_{1}1)\big[(\forall^{\st}n)(\exists \beta^{0})&(|\beta|=n \wedge \beta\in T ) \di (\exists^{\st}\alpha^{1}\leq_{1}1)(\forall^{\st}n^{0})(\overline{\alpha}n\in T)   \big]
\end{align}
where `$T\leq_{1}1$' denotes that $T$ is a binary tree.  
Clearly, \eqref{fanns} is a nonstandard version of \emph{weak K\"onig's lemma}, and the latter is a compactness principle by \cite{simpson2}*{IV}. 

\medskip

As it happens, $\STP$ and $\STP_{\R}$ express the nonstandard compactness of Cantor space and the unit interval by \emph{Robinson's theorem} in \cite{loeb1}*{p.\ 43}.  
Furthermore, \eqref{fanns} is equivalent to the following normal form:
\begin{align}\label{frukkklk}
(\forall^{\st}g^{2})(\exists^{\st}w^{1^{*}})(\forall T^{1}\leq_{1}1)(\exists ( \alpha^{1}\leq_{1}1,  &~k^{0}) \in w)\big[(\overline{\alpha}g(\alpha)\not\in T)\\
&\di(\forall \beta\leq_{1}1)(\exists i\leq k)(\overline{\beta}i\not\in T) \big] \notag
\end{align}
as shown below in Theorem \ref{klak}.  Term extraction as in Theorem \ref{consresultcor} converts \eqref{frukkklk} to the following `special fan functional'.
\bdefi[Special fan functional]
\begin{align}\label{fanns333}\tag{$\SCF(\Theta)$}
(\forall g^{2}, T^{1}\leq_{1}1)\big[ (\forall  \alpha^{1}\in \Theta(g)(2))(\alpha\leq_{1}1\di& \overline{\alpha}g(\alpha)\not\in T)\di\\
 &(\forall \beta\leq_{1}1)(\exists i\leq_{0}\Theta(g)(1))(\overline{\beta}i\not\in T) \big].\notag
\end{align}
\edefi
Any functional $\Theta^{2\di (1^{*}\times 0)}$ satisfying $\SCF(\Theta)$ is called a \emph{special fan functional} and the latter object was first introduced in \cite{samGH}.  
Note that there is no unique such $\Theta$, i.e.\ it is in principle incorrect to talk about `the' special fan functional.  

\medskip

The computability-theoretic properties of the special fan functional are studied in \cites{samGH, dagsam}. Intuitionistically, a special fan functional $\Theta$ can be computed (via a term in G\"odel's ${T}$; see Theorem \ref{koreki}) in terms of the \emph{intuitionistic fan functional} $\Omega$ as in $\MUC(\Omega)$ (See e.g.\ \cite{kohlenbach2}*{\S3} for the latter) defined as follows:
\be
(\forall Y^{2}) (\forall f^{1}, g^{1}\leq_{1}1)(\overline{f}\Omega(Y)=\overline{g}\Omega(Y)\notag \di Y(f)=Y(g)). \label{lukl3}\tag{$\textsf{\textup{MUC}}(\Omega)$}
\ee
Classically, $\Theta$ can be computed (Kleene's S1-S9 from \cite{longmann}*{\S5.1.1}) by $\xi$ as in $(\mathcal{E}_{2})$: 
\be\tag{$\mathcal{E}_{2}$}\label{hah}
(\exists \xi^{3})(\forall Y^{2})\big[  (\exists f^{1})(Y(f)=0)\asa \xi(Y)=0  \big].
\ee
but $\Theta$ cannot be computed (Kleene S1-S9) from any type two functional, which includes rather non-computable functionals like $(\mu^{2})$ and the Suslin functional.  
With regard to first order-strength, $\RCAo+(\exists \Theta)\SCF(\Theta)$ is a conservative extension of $\RCA_{0}^{2}+\WKL$ by \cite{kohlenbach2}*{Prop.\ 3.15}.  
All these results may be found in \cite{dagsam}.

\medskip

In conclusion, the special fan functional has rather weak first-order strength, while it is extremely hard to compute (compared to e.g.\ the computational strength of the Big Five systems of Reverse Mathematics).    As will become clear in Section~\ref{SSTP}, the special fan functional still provides plenty of computational content after applying an extra (algorithmic) step.

\section{The intermediate value theorem}\label{frastruc}
We discuss Wattenberg's treatment from \cite{watje}*{II} of the intermediate value theorem inside Nonstandard Analysis.  
\subsection{Introduction and preliminaries}
The intermediate value theorem $(\IVT)$ is a basic result from calculus and is usually formulated as follows.  
\begin{thm}[$\IVT$]
Suppose that $f:[a,b]\di \R$ is continuous and such that $f(a)<_{\R}0<_{\R}f(b)$.  Then there is $x\in [a,b]$ such that $f(x)=_{\R}0$.
\end{thm}
As is well-known, $\IVT$ implies a non-trivial fragment of the law of excluded middle (See e.g.\ \cite{beeson1}*{I.7} or \cite{mandje}) and is therefore rejected in constructive mathematics.
The following `approximate' $\IVT$ is described as `constructive' by Wattenberg (\cite{watje}*{II.2}).
\begin{thm}[$\IVT_{\bullet}$]
Suppose that $\eps>_{\R}0$ and $f:[a,b]\di \R$ is continuous and such that $f(a)<_{\R}0<_{\R}f(b)$.  Then there is $x\in [a,b]$ such that $|f(x)|<_{\R}\eps$.
\end{thm}
On one hand, a slight modification of Wattenberg's nonstandard proof of $\IVT$ will be shown to yield the following \emph{effective} version of $\IVT$ in Section~\ref{good}.  
\begin{thm}[$\IVT_{\ef}(s)$]
For $k^{0}$ and $f:[a,b]\di \R$ uniformly continuous with modulus $g$ and such that $f(a)<_{\R}0<_{\R}f(b)$, we have $|f(s(f,g,k))|<_{\R}\frac{1}{k}$.
\end{thm}
Note that this version no longer involves Nonstandard Analysis.  Furthermore, the term $s$ is `read off' from the (modified) Wattenberg proof.  
Thus, Wattenberg's claims about the effective content of Nonstandard Analysis do seem to hold water.  

\medskip

One the other hand, some of Wattenberg's (explicit and implicit) claims regarding the constructive status of the nonstandard axioms \emph{Transfer} and \emph{Standard Part} will be shown to be incorrect.  
In particular, Wattenberg suggests that \emph{Standard Part} is fundamentally non-constructive (See \cite{watje}*{p.\ 303} and Section \ref{intro}), 
but freely (either explicitly or implicitly) makes use of \emph{Transfer}.  We show that \emph{Transfer} as used by Wattenberg is fundamentally non-constructive in Section \ref{good2}, while the constructive status of \emph{Standard Part} is discussed in Section \ref{bad}.

\subsection{Constructive $\IVT$ and Nonstandard Analysis}\label{good}
Wattenberg proves the classical $\IVT$ inside Nonstandard Analysis in \cite{watje}*{II.3} using the following steps:
\begin{enumerate}  
\renewcommand{\theenumi}{\roman{enumi}}
\item Define $x_{j}:=a+jh$ for $h:=\frac{b-a}{N}$ where $N^{0}$ is a nonstandard number.  \label{koer}
\item Let $j_{0}$ be the largest $j$ such that $f(x_{j})\leq_{\R} 0$.\label{frag}
\item Since $f(x_{j_{0}})\leq_{\R} 0\leq_{\R}f(x_{j_{0}+1})$ and $x_{j_{0}}\approx x_{j_{0}+1}$, we have $f(x_{j_{0}})\approx 0$. \label{cruxjes}
\item Let standard $t_{0}\in [0,1]$ be such that $t_{0}\approx x_{j_{0}}$ and conclude $f(t_{0})=_{\R}0$.  \label{noco}
\end{enumerate}
Note that item \eqref{cruxjes} makes use of \emph{uniform} nonstandard continuity, while item \eqref{noco} makes use of $\STP$ to obtain $t_{0}$.
According to Wattenberg (\cite{watje}*{p.\ 303}), the final step of the proof is non-constructive (as it involves $\STP$) and should therefore be omitted.  
Following this approach, the previous steps immediately yield Theorem~\ref{wakko} where $t^{((1\di 1)\times 0)\di 0}$ is defined as follows (with the real $h:=\frac{b-a}{N}$):
\be\label{teke2}
t(f,N):=
\begin{cases}
(\mu j \leq N)\big(\big[f\big(jh\big)\big](2^{N})\leq_{0} 0\big)h & \textup{if such exists}\\
N+1 & \textup{otherwise}
\end{cases}.
\ee
\begin{thm}[$\IVT_{\textsf{\wat}}$]\label{wakko}
For nonstandard $N^{0}$ and nonstandard uniformly continuous $f:[a,b]\di \R$ such that $f(a)<_{\R}0<_{\R}f(b)$, $t(f,N)\in [a,b]\wedge f(t(f, N))\approx 0$.  
\end{thm}
Note that $t$ does not implement item \eqref{frag}, but only an \emph{approximation} up to nonstandard precision.  This is however \emph{equally good} in light of nonstandard (uniform) continuity from item \eqref{cruxjes}.  It should be noted that Wattenberg observes this problem with item \eqref{frag} as well in \cite{watje}*{p.\ 304}.  He proposes to replace `$\leq_{\R}$' by `$\lessapprox$' in \cite{watje}*{II.6}, which is however not a decidable\footnote{It should be noted that the constructive system $\H$ proves the `lesser limited principle of omniscience' $\LLPO$ relative to `st' (\cite{brie}*{\S3.1}) which implies $(\forall^{\st} x\in \R)(x\lessapprox 0 \vee x\gtrapprox 0)$, i.e.\ an instance of the law of excluded middle relative to `st'.  
However, the decision procedure for the latter follows easily from the constructive fact that $(\forall x\in \R)(x\geq_{\R}0 \vee x\leq_{\R}a)$ for $0<_{\R}a\approx 0$.} predicate either.  

\medskip

We have the following theorem where it should be noted that $\IVT_{\wat}$ does not have an obvious normal form (to apply Theorem \ref{consresultcor}).  
\begin{thm}\label{hark}
From the proof of $\IVT_{\textsf{\wat}}$ in $\P_{0}$, a term $u$ can be extracted such that $\RCAo$ proves $\IVT_{\ef}(u)$.
\end{thm}
\begin{proof}
The proof of $\IVT_{\wat}$ in $\P_{0}$ follows immediately from the definition of $t$ in \eqref{teke} and steps \eqref{koer}-\eqref{cruxjes} above.  
We now convert $\IVT_{\wat}$ into a normal form so that we can apply Theorem \ref{consresultcor}; we assume $a=0$ and $b=1$ for simplicity.  
Thus, let $A(f, k, N)$ be the conjunction of ${ f(0)<_{\R}0<_{\R}f(1)}$ and the formula in square brackets in \eqref{soareyou2}, and note that $(\forall^{\st}k^{0})(\exists^{\st}N^{0})A(f, k,N)$ implies that $f$ is continuous as in \eqref{soareyou4}.  In this way, $\IVT_{\wat}$ implies that for $f:[0,1]\di \R$ we have
\[
(\forall^{\st}k^{0})(\exists^{\st}M^{0})A(f, k,M) \di (\forall N\in \Omega )(t(f,N)\in [0,1]\wedge f(t(f, N)h/N)\approx 0),
\]
which immediately yields by resolving `$\approx$' that 
\[\textstyle
(\forall^{\st}k^{0})(\exists^{\st}M^{0})A(f, k,M) \di (\forall^{\st}l^{0})(\forall N\in \Omega )\big[t(f,N)\in [0,1]\wedge |f(t(f, N)h/N)|<\frac{1}{l}\big],
\]
where we abbreviate the formula in square brackets by $B(f, N, l)$.  Applying underspill to $(\forall N\in \Omega )B(f, N, l)$, we obtain $(\forall^{\st}l^{0})(\exists^{\st}n^{0})(\forall N\geq n )B(f, N, l)$.  Thus, $\IVT_{\wat}$ implies that for all $f:[0,1]\di \R$ we have
\be\label{konkey}
 \big[(\forall^{\st}k)(\exists^{\st}M)A(f, k,M) \di   (\forall^{\st}l)(\exists^{\st}n)(\forall N\geq n )B(f, N, l)\big], 
\ee
and, since standard functionals yield standard output on standard input, we have for all $f:[0,1]\di \R$ that
\be\label{kakoo}
(\forall^{\st}g^{1}) \big[(\forall^{\st}k^{0})A(f, k,g(k)) \di   (\forall^{\st}l^{0})(\exists^{\st}n^{0})(\forall N\geq n )B(f, N, l)\big].
\ee
By strengthening the antecedent, the previous formula yields:
\be\label{wakoo}
(\forall^{\st}g)(\forall f:[0,1]\di \R) \big[(\forall k)A(f, k,g(k)) \di   (\forall^{\st}l)(\exists^{\st}n)(\forall N\geq n )B(f, N, l)\big].  
\ee
Now bring outside all standard quantifiers as far as possible to obtain:
\be\label{wakoo2}
(\forall^{\st}g, l)(\forall f:[0,1]\di \R) (\exists^{\st}n)\big[(\forall k^{0})A(f, k,g(k)) \di   (\forall N\geq n )B(f, N, l)\big].  
\ee
Let $C(g, l, f, n)$ be the formula in square brackets in the previous formula.  
Apply \emph{Idealisation} \textsf{I}, while bearing in mind Remark \ref{simply}, to push the standard quantifier involving $n$ to the front as follows:
\be\label{cruxi}
(\forall^{\st}g^{1}, l^{0}) (\exists^{\st}n^{0})(\forall f:[0,1]\di \R)C(g, l, f, n). 
\ee
Now apply Theorem \ref{consresultcor} to `$\P\vdash\eqref{cruxi}$' to obtain a term $s$ such that $\EPA^{\omega*}$ proves
\be\label{cruxi2}
(\forall^{\st}g^{1}, l^{0}) (\exists n^{0}\in s(g, l))(\forall f:[0,1]\di \R)C(g, l, f, n). 
\ee
Define $u(f, g, k):=t(f, \max_{i<|s(g, k)|}s(g, k)(i))$ and note that $\IVT_{\ef}(u)$.  
\end{proof}
The previous proof serves as a template for obtaining computational content from Nonstandard Analysis as follows.  
The below proofs follow this template and we will therefore not always go in as much detail as in the previous proof.  
\begin{tempie}[Computational content of Nonstandard Analysis]\label{tempe}\rm~
\begin{enumerate}
\renewcommand{\theenumi}{\roman{enumi}}
\item Bring all sub-formulas into a normal form like in \eqref{konkey}.
\item Push all standard quantifiers to the front as follows:  
\begin{enumerate}
\item[(ii.a)] If necessary, introduce standard functionals like $g$ in \eqref{kakoo} and drop `st' in the antecedent like in \eqref{wakoo}. 
\item[(ii.b)] If necessary, use \emph{Idealisation} (like for \eqref{wakoo2}) to pull standard quantifiers through normal quantifiers.  
\end{enumerate}
\item Obtain a normal form like \eqref{cruxi} and apply Theorem \ref{consresultcor} using Remark \ref{simply}.  
\end{enumerate}
\end{tempie}
Furthermore, there is a subtlety involved in the formulation of $\IVT_{\wat}$ as follows.  
\begin{rem}[On $\Omega$-invariance]\label{notnotnting}\rm
By definition, $t(f, N)$ from $\IVT_{\wat}$ is such that $f(t(f, N))\approx 0$ for any $N\in \Omega$.  However, if $f$ has multiple intermediate values, i.e.\ there are $x, y\in [0,1]$ such that $f(x)\approx f(y)\approx 0$ but $x\not \approx y$, then it is possible that $t(f, N)\not\approx t(f, M)$ for $N, M\in \Omega$.  In other words, we cannot use $\Omega$\textsf{-CA} to obtain a \emph{standard} intermediate value of $f$.  
\end{rem}
As it turns out, the proof of the theorem also goes through constructively, as follows.  
This is far from obvious as the proof seems to involve non-constructive steps like the independence of 
premise principle to go from \eqref{wakoo} to \eqref{wakoo2}.  
\begin{cor}
From the proof of $\IVT_{\textsf{\wat}}$ in $\H$, a term $u$ can be extracted such that $\textsf{\textup{E-HA}}^{\omega*}$ proves $\IVT_{\ef}(u)$.
\end{cor} 
\begin{proof}
Clearly, the above proof of $\IVT_{\wat}$ goes through in $\H$.  Furthermore, one easily derives \eqref{cruxi2} from $\IVT_{\wat}$ in $\H$.  Indeed, the steps leading up to \eqref{wakoo} clearly go through in $\H$.  For the step from \eqref{wakoo} to \eqref{wakoo2}, the `$(\forall^{\st}l)$' quantifier can be brought to the front in intuitionistic logic, and the same can then be done for the quantifier `$(\exists^{\st}n)$' using the axiom $\HIP_{\forall^{\st}}$ from Definition \ref{flah}, while bearing in mind Remark \ref{simply} as usual.  Having obtained \eqref{wakoo2}, one applies $\NCR$ to obtain \eqref{cruxi}, again bearing in mind Remark \ref{simply}.  Finally, one applies Theorem \ref{consresult2} to `$\H\vdash \eqref{cruxi}$' to obtain the required term.  
\end{proof}
In conclusion, Wattenberg's claim that Nonstandard Analysis has effective (even constructive) content seems correct in light of Theorem \ref{hark} and its corollary.

\subsection{Non-constructivity arising from continuity}\label{good2}
In this section, we deal with the exact connection between nonstandard and $\eps$-$\delta$-continuity.  
We are motivated by the observation that Wattenberg 
uses the (provable using \emph{Transfer}) equivalence between nonstandard and $\eps$-$\delta$-continuity without a second thought in \cite{watje}*{III.3}.  

\medskip

By the following theorem and Theorem \ref{klabaka}, any step from $\eps$-$\delta$-continuity to nonstandard continuity, i.e.\ $\NSC_{1}$ and $\NSC_{2}$ in Definition \ref{mikeh}, implies a non-trivial fragment of \emph{Transfer} and is therefore fundamentally \emph{non-constructive}.  
\begin{thm}\label{hark2} 
The systems $\P_{0}+\NSC_{2}$ and $\P+\NSC_{1}$ both prove $\paai$.  
\end{thm}
\begin{proof}
For the first part, we work in $\P_{0}+\NSC_{2}$.  Thus, fix standard $f_{0}\in C^{\st}([0,1])$ and suppose $\neg\paai$, i.e.\ there is standard $h_{0}^{1}$ such that $(\forall^{\st}n)h_{0}(n)=0$ and $(\exists m_{0})h(m_{0})\ne0$.  Let $\mathbb{b}^{1}$ be such that $\mathbb{b}(q)=q$ if $q^{0}\ne 0$ and $1$ otherwise.  
Now define standard $f_{1}$ as follows: $f_{1}(x)=f_{0}(x)$ if $\big(\forall n^{0}\leq \frac{1}{\mathbb{b}([x](1))}\big)(h_{0}(n)=0)$, and $f_{0}(x)+1$ otherwise.  
Since for a standard real $x\in[0,1]$, the rational $[x](n)$ is standard for standard $n^{0}$, we have $(\forall^{\st}x\in [0,1])(f_{0}(x)=_{\R}f_{1}(x))$, and hence $f_{1}\in C^{\st}([0,1])$ by definition.  
However, $f_{1}$ is not nonstandard continuous since:
\[\textstyle
f_{1}(0)=_{\R}f_{0}(0)\not\approx f_{0}(0)+1\approx f_{0}(\frac{1}{2^{N}})+1=_{\R}f_{1}(\frac{1}{2^{N}})
\]
for large enough nonstandard $N^{0}$.  This contradiction finishes the first part.

\medskip

For the second part, we work in $\P_{0}+\NSC_{1}$.  Suppose $\neg\paai$, i.e.\ there is standard $h_{0}^{1}$ such that $(\forall^{\st}n)h_{0}(n)=0$ and $(\exists m_{0})h(m_{0})\ne0$.  
Define the standard real $x_{0}$ as $\sum_{n=0}^{\infty}\frac{h(n)}{2^{n}}$.  Since $0\approx x_{0}>_{\R}0$ the standard function $f_{2}(x):=\frac{1}{|x|+x_{0}}$ is clearly well-defined and continuous (as in $f_{2}\in C([0,1])$).  However, $f_{2}(x_{0})=\frac{1}{2x_{0}}\not\approx \frac{1}{x_{0}}=_{\R} f_{2}(0)$ implies that $f_{2}$ is not nonstandard continuous.  This contradiction yields $\paai$, and we are done.        
\end{proof}
Recall the definition of `modulus-of-continuity-functional' from Definition \ref{mikeh}. 
\begin{cor}\label{idare}
From the proof that $\P_{0}\vdash \NSC_{1}\di \paai$, a term $t$ can be extracted such that $\RCAo\vdash (\forall \Xi^{3})\big(\MPC(\Xi)\di \MU(t(\Xi)))$
\end{cor}
\begin{proof}
A normal form for $\paai$ is given by \eqref{geluklkt}, where we use $A(f, n)$ to denote the formula in square brackets.  
A normal form for nonstandard pointwise continuity \eqref{soareyou3} is obtained as follows.  
Resolving `$\approx$' in \eqref{soareyou3}, we obtain  
\[\textstyle
(\forall^{\st}x\in [0,1])(\forall y\in [0,1])\big( (\forall^{\st}N) (|x- y|<\frac{1}{N}) \di (\forall^{\st}k )(|f(x)- f(y)|<\frac{1}{k})\big).  
\]
We may bring out the `$(\forall^{\st}k)$' and `$(\forall^{\st}N) $' quantifiers as follows:
\[\textstyle
(\forall^{\st}x\in [0,1])(\forall^{\st}k)\underline{(\forall y\in [0,1])(\exists^{\st}N)\big( |x- y|<\frac{1}{N} \di|f(x)- f(y)|<\frac{1}{k}\big)}.  
\]
Applying \emph{Idealisation} \textsf{I} to the underlined formula, we obtain a standard $z^{0^{*}}$ such that $(\forall y\in [0,1])(\exists N\in z)$ in the previous formula.  
Now let $N_{0}$ be the maximum of all numbers in $z$, and note that for $N=N_{0}$, we have the following:
\[\textstyle
(\forall^{\st}x\in [0,1])(\forall^{\st}k)(\exists^{\st}N)(\forall y\in [0,1])\big( |x- y|<\frac{1}{N} \di|f(x)- f(y)|<\frac{1}{k}\big), 
\]
abbreviated by $(\forall^{\st}x\in [0,1],k)(\exists^{\st}N)B(x, k, N, f)$.  Hence, $\NSC_{1}\di \paai$ is  
\[
(\forall^{\st}g\in C([0,1]))(\forall^{\st}x\in [0,1], k)(\exists^{\st}N)B(x, k, N, g)\di (\forall^{\st} f^{1})(\exists^{\st}n^{0})A(f,n), 
\]
which implies (since standard functionals have standard output for standard input):
\[
(\forall^{\st}\Psi)\big[(\forall^{\st}g\in C([0,1]))(\forall^{\st}x\in [0,1], k)B(x, k, \Psi(x, k, g), g)\di (\forall^{\st} f^{1})(\exists^{\st}n^{0})A(f,n)\big], 
\]
and dropping the `st' in the antecedent, we obtain:
\[
(\forall^{\st}\Psi)\big[(\forall g\in C([0,1]), x\in [0,1], k)B(x, k, \Psi(x, k, g), g)\di (\forall^{\st} f^{1})(\exists^{\st}n^{0})A(f,n)\big], 
\]
and bringing the standard quantifiers up front, we finally have
\be\label{krefke}
(\forall^{\st}\Psi, f)(\exists^{\st}n)\big[(\forall g\in C([0,1]), x\in [0,1], k)B(x, k, \Psi(x, k, g), g)\di A(f,n)\big].
\ee
Applying Theorem \ref{consresultcor} to `$\P_{0}\vdash \eqref{krefke}$', we obtain a term $t$ such that
\[
(\forall \Psi, f)(\exists n\in t(\Psi, f))\big[(\forall g\in C([0,1]), x\in [0,1], k)B(x, k, \Psi(x, k, g), g)\di A(f,n)\big].
\]
which is exactly as required in light of the definition of $A, B$ and Remark \ref{simply}.
\end{proof}
By the previous theorem and corollary, $\NSC_{1}$ translates to the existence of a modulus-of-continuity-functional when applying Theorem \ref{consresultcor}.  
Such a functional is fundamentally non-constructive by the corollary, and this non-constructiveness `trickles down' to any nonstandard theorem of $\P_{0}+\NSC_{1}$ as follows.
\begin{cor}\label{idare56} Let $\varphi$ be internal.
From $\P_{0}+ \NSC_{1}\vdash (\forall^{\st}x)(\exists^{\st}y)\varphi(x, y)$, a term $t$ can be extracted such that $\RCAo\vdash (\forall \Xi^{3})\big(\MPC(\Xi)\di  (\forall x)(\exists y\in t(x, \Xi))\varphi(x, y)  \big)$.
\end{cor}
\begin{proof}
Analogous to the proof of the previous corollary following Template \ref{tempe}  
\end{proof}
The previous theorem and corollaries imply that the step from $\eps$-$\delta$ continuity (relative to `st' or not) to the nonstandard variety always involves a non-trivial instance of \emph{Transfer}, which is fundamentally non-constructive.  In particular, by Corollary \ref{idare56}, any result proved using $\NSC_{1}$ only provides computational information \emph{involving a non-constructive modulus-of-continuity-functional}.  
In general, moving from the standard into the nonstandard world is highly non-constructive (requiring \emph{Transfer}), as was sketched in Section \ref{intro} in the form of Osswald's \emph{local constructivity}.     
Nonetheless, Wattenberg freely uses \emph{Transfer} and the equivalence between `$\eps$-$\delta$' and nonstandard continuity in \cite{watje}*{II-III}.  
This aspect of his investigation into the computational content of Nonstandard Analysis thus seems incorrect.  

\medskip

%
%
Furthermore, as suggested by its proof, Theorem \ref{hark2} goes through for other notions besides continuity.  
We now show, for differentiability and Riemann integration, that the step from the $\eps$-$\delta$ definition to the nonstandard one yields $\paai$.  

\medskip

First of all, we have the usual definition of differentiability.  
\bdefi A function $f$ is \emph{nonstandard differentiable} at $a$ if
\begin{equation}\label{soareyou22}\textstyle
(\forall \eps, \eps' \ne0)\big(\eps, \eps'\approx 0 \di \frac{f(a+\eps)-f(a)}{\eps}\approx \frac{f(a+\eps')-f(a)}{\eps'}\big).
\end{equation}
A function $f$ is \emph{differentiable} at $a$ if
\begin{equation}\label{soareyou222}\textstyle
(\forall k^{0})(\exists N^{0})(\forall \eps, \eps')\big(0<|\eps|, |\eps'|<\frac{1}{N} \di
 \big|\frac{f(a+\eps)-f(a)}{\eps}- \frac{f(a+\eps')-f(a)}{\eps'}\big|<\frac1k\big).
\end{equation}
A `modulus of differentiability at $a$' is a function $g^{1}$ such that $g(k)$ is $N^{0}$ in \eqref{soareyou222}.  
\edefi
Let $\NSD$ be the statement \emph{any standard $f:\R\di \R$ differentiable at zero is also {nonstandard differentiable} there}.  
Now, $\NSD$ is a theorem of $\IST$ but we also have the following implication.  
\begin{thm}\label{hark3} 
The system $\P+\NSD$ proves $\paai$.  
\end{thm}
\begin{proof}
Working in $\P+\NSD$, suppose we have $\neg\paai$, i.e.\ there is standard $h_{0}^{1}$ such that $(\forall^{\st}n)h_{0}(n)=0$ and $(\exists m_{0})h(m_{0})\ne0$.  
Define the standard real $x_{0}$ as in the proof of Theorem \ref{hark2}.  Since $0\approx x_{0}>_{\R}0$ the standard function $f_{0}(x):=e^{\frac{1}{x^{2}+x_{0}}}$ is well-defined and differentiable in the usual internal `epsilon-delta' sense.  However, 
\[\textstyle
\frac{f_{0}(\sqrt{x_{0}})-f_{0}(0)}{\sqrt{x_{0}}}= \frac{e^{\frac{1}{2x_{0}}}-e^{\frac{1}{x_{0}}}}{\sqrt{x_{0}}}= e^{\frac{1}{x_{0}}}\frac{e^{\frac{1}{x_{0}}}-1}{\sqrt{x_{0}}}\gg 0 \gg  e^{\frac{1}{x_{0}}}\frac{e^{\frac{1}{x_{0}}}-1}{-\sqrt{x_{0}}}  =\frac{e^{\frac{1}{2x_{0}}}-e^{\frac{1}{x_{0}}}}{-\sqrt{x_{0}}}=\frac{f_{0}(-\sqrt{x_{0}})-f_{0}(0)}{-\sqrt{x_{0}}}
\]
which implies that $f_{0}$ is not nonstandard differentiable at zero.  
This contradiction yields $\paai$, and we are done.        
\end{proof}
Let $\textsf{DIF}(\Xi)$ be the statement that $\Xi(f)$ is a modulus for differentiability at zero for every $f$ differentiable at zero.  
\begin{cor}\label{idarecor}
From the proof that $\P\vdash \NSD\di \paai$, a term $t$ can be extracted such that $\textsf{\textup{E-PA}}^{\omega}\vdash (\forall \Xi^{3})\big(\textsf{\textup{DIF}}(\Xi)\di \MU(t(\Xi)))$
\end{cor}
\begin{proof}
A normal form for differentiability as in \eqref{soareyou22} is easy to obtain and as follows:  
\[\textstyle
(\forall^{\st}k^{0})(\exists^{\st} N^{0})(\forall \eps, \eps' \ne0)\big(|\eps|, |\eps'| <\frac{1}{N} \di \left|\frac{f(a+\eps)-f(a)}{\eps}- \frac{f(a+\eps')-f(a)}{\eps'}\right|<\frac{1}{k}\big),
\]
The proof is straightforward and analogous to the proof of Corollary \ref{idare}.  
\end{proof}
Hence, switching from epsilon-delta differentiability to the nonstandard variety as in $\NSD$ is at least as non-constructive as $(\mu^{2})$ and $\NSC_{1}$.  
One readily obtains a version of $\NSC_{2}\di \paai$ for $\NSD$, i.e.\ for $\eps$-$\delta$-differentiability relative to `st'.  

\medskip

%
Next, we consider the usual definitions of Riemann integration.  
\bdefi[Riemann Integration]\label{kunko}~
\begin{enumerate}
\item A \emph{partition} of $[0,1]$ is an increasing sequence $\pi=(0, t_{0}, x_{1},t_{1},  \dots,x_{M-1}, t_{M-1}, 1)$.  We write `$\pi \in P([0,1]) $' to denote that $\pi$ is such a partition.
\item For $\pi\in P([0,1])$, $\|\pi\|$ is the \emph{mesh}, i.e.\ the largest distance between two adjacent partition points $x_{i}$ and $x_{i+1}$. 
\item For $\pi\in P([0,1])$ and $f:\R\di \R$, the real $S_{\pi}(f):=\sum_{i=0}^{M-1}f(t_{i}) (x_{i+1}-x_{i}) $ is the \emph{Riemann sum} of $f$ and $\pi$.  
\item A function $f$ is \emph{nonstandard integrable} on $[0,1]$ if
\be\label{soareyou5}
(\forall \pi, \pi' \in P([0,1]))\big[\|\pi\|,\| \pi'\|\approx 0  \di S_{\pi}(f)\approx S_{\pi'}(f)  \big].
\ee
\item A function $f$ is \emph{integrable} on $[0,1]$ if
\be\label{soareyou6}\textstyle
(\forall k^{0})(\exists N^{0})(\forall \pi, \rho \in P([0,1]))\big[\|\pi\|,\| \rho\|<\frac{1}{N}  \di |S_{\pi}(f)- S_{\rho}(f)|<\frac{1}{k}  \big].
\ee
A modulus of (Riemann) integration $\omega^{1}$ provides $N=\omega(k)$ as in \eqref{soareyou6}.  
\end{enumerate}
\edefi
Let $\NSR$ be the statement \emph{a standard $f:\R\di \R$ integrable on the unit interval is also \emph{nonstandard integrable} there}.  
As above, $\NSR$ is a theorem of $\IST$ but we also have the following implication.  
\begin{thm}\label{hark38} 
The system $\P+\NSR$ proves $\paai$.  
\end{thm}
\begin{proof}
Suppose $\NSR\wedge \neg \paai$ and note that $f_{0}$ from Theorem \ref{hark3} is Riemann integrable.  However, since the distance between $f_{0}(0)=e^{\frac{1}{x_{0}}}$ and $f_{0}(\sqrt{x_{0}})=e^{\frac{1}{2x_{0}}}$ is larger than any standard real, replacing $0$ by $\sqrt{x_{0}}$ in a partition causes the associated Riemann sums to be 
apart by more than an infinitesimal. \
\end{proof}
Let $\textsf{RIE}(\kappa)$ be the statement that $\kappa(f)$ is a modulus of Riemann integration for every $f:\R\di \R$ integrable on the unit interval.  
\begin{cor}\label{idare22}
From the proof that $\P\vdash \NSR\di \paai$, a term $t$ can be extracted such that $\textsf{\textup{E-PA}}^{\omega}\vdash (\forall \kappa^{3})\big(\textsf{\textup{RIE}}(\kappa)\di \MU(t(\kappa)))$
\end{cor}
\begin{proof}
The proof is straightforward and analogous to that of Corollary \ref{idare}. 
\end{proof}
Hence, switching from epsilon-delta integrability to the nonstandard variety as in $\NSR$ is at least as non-constructive as $(\mu^{2})$ and $\NSC_{1}$. 
One readily obtains a version of $\NSC_{2}\di \paai$ for $\NSR$, i.e.\ for $\eps$-$\delta$-integrability relative to `st'.  
%

\medskip

In conclusion, Wattenberg's claim that Nonstandard Analysis has effective (even constructive) content is correct in light of Theorem \ref{hark}; his implementation using $\eps$-$\delta$ continuity and \emph{Transfer} is \emph{problematic} in light of Theorem~\ref{hark2}, but \emph{easily salvageable}: By Theorem \ref{hark} it suffices to just adopt nonstandard (rather than $\eps$-$\delta$) continuity, in line with Osswald's local constructivity.  
We next investigate Wattenberg's claims regarding the constructive status of \emph{Standard Part} in Section~\ref{bad}.

\subsection{The non-constructive status of Standard Part}\label{bad}
We investigate the constructive status of \emph{Standard Part} in light of Wattenberg's claims that it be fundamentally non-constructive.  

\medskip

First of all, we have the following theorem regarding the use of $\STP$, which is the only fragment of \emph{Standard Part} used by Wattenberg in \cite{watje}.  
\begin{thm}\label{daradal} Let $\varphi$ be internal.
From $\P_{0}+ \STP\vdash (\forall^{\st}x)(\exists^{\st}y)\varphi(x, y)$, a term $t$ can be extracted such that $\RCAo\vdash (\forall \Theta^{3})\big(\SCF(\Theta)\di  (\forall x)(\exists y\in t(x, \Theta))\varphi(x, y)  \big)$.
\end{thm}
\begin{proof}
Note that $\STP$ is equivalent to \eqref{frukkklk} form Section \ref{firstkol} by Theorem \ref{klak}.  The theorem now follows easily by following the proof of Corollary \ref{idare}.    
\end{proof}
In light of the previous theorem, the use of $\STP$ in the proof of a nonstandard theorem translates to the presence of the \emph{special fan functional} $\Theta$ after applying Theorem \ref{consresultcor}.  
Given the computational hardness of $\Theta$, Wattenberg's claims regarding the non-constructive nature of Standard Part seem justified.  
However, as will be established in Section \ref{SSTP}, the fragment of \emph{Standard Part} used by Wattenberg (namely $\STP$ from Section \ref{firstkol}) does have plenty 
of effective content, though extra technical machinery is needed for this.  

\medskip

Secondly, we show that the generalisation of $\STP$ to type two functionals is non-constructive as it implies $(\exists^{2})$.   
In particular, the following rather weak fragment of \emph{Standard part} is established to be non-constructive by Theorem \ref{konky}. 
%
\be\tag{$\STP_{2}$}
(\forall Y^{2}\leq_{2}1)(\exists^{\st}Z^{2}\leq_{2}1)(Z\approx_{2} Y).
\ee  
Note that $(\textsf{\textup{E}})^{\st}_{n+2}$ results in a conservative extension of $\P_{0}$ as shown in \cite{bennosam}. 
\begin{thm}\label{konky}
The system $\P_{0}+(\textsf{\textup{E}})^{\st}_{2}+\STP_{2}$ proves $(\exists^{2})^{\st}$.  
\end{thm}
\begin{proof}
In a nutshell, we work relative to `st' in $\P_{0}+(\textsf{\textup{E}})^{\st}_{2}+\STP_{2}$ and define a functional which computes the separating set in $\Sigma_{1}^{0}$-separation (\cite{simpson2}*{I.11.7}).
The theorem then follows from the equivalence between the uniform version of $\Sigma_{1}^{0}$-separation and $(\exists^{2})$, as proved in \cite{yamayamaharehare}*{Theorem 3.6}.  

\medskip

Let $f_{1}, f_{2}$ be standard binary sequences and fix nonstandard $N^{0}$.  Let $K(n, f_{1}, f_{2})$ be the largest $ k^{0}\leq N$ such that $(\forall n_{1}, n_{2}\leq k)(f_{1}(n_{1}, n)\ne 0 \vee f_{2}(n_{2}, n)\ne 0)$, if such number exists, and zero otherwise.  Define $Y^{2}$ as follows:
\[
Y(f_{1}, f_{2}, n):=
\begin{cases}
1 &  (\exists n_{1}\leq K(n, f_{1}, f_{2}))(f_{1}(n_{1}, n) = 0 )  \\
0 & \textup{otherwise}
\end{cases}.
\]
Now suppose $(\forall^{\st}n^{0})(\neg\varphi^{\st}_{1}(n)\vee \neg\varphi^{\st}_{2}(n))$ where $\varphi_{i}(n)\equiv (\exists n_{i})(f(n_{i}, n)=0)$.  
By overspill, $K(n, f_{1}, f_{2})$ is nonstandard for all standard $n^{0}$.  
By definition, we have
\[
(\forall^{\st}n^{0})\big[ \varphi^{\st}_{1}(n)\di Y(f_{1}, f_{2}, n)= 1\wedge  \varphi^{\st}_{2}(n)\di Y(f_{1}, f_{2}, n)= 0  \big].
\]
Now apply $\STP_{2}$ to obtain standard $Z^{2}$ such that $Z\approx_{2} Y $.   Then $Z(f_{1}, f_{2}, n)$ is standard and provides the separating set from $\Sigma_{1}^{0}$-separation for standard inputs and relative to `st'.     
\end{proof}
We provide an alternative proof for Theorem \ref{konky} as follows.  
\begin{proof}
We prove that $\STP_{2}\di \UWKL^{\st}$ in $\P$, where $\UWKL$ is as follows:
\be
(\exists \Phi^{1\di 1})(\forall T^{1}\leq_{1}1)\big[ (\forall n^{0})(\exists \beta^{0})(\beta\in T\wedge |\beta|=n)\di (\forall m^{0})(\overline{\Phi(T)}m\in T)  \big].
\ee
As proved in \cite{kooltje}, $\UWKL$ implies $(\exists^{2})$, and the latter proof immediately transfers to $\P_{0}+(\textsf{\textup{E}})^{\st}_{2}$, yielding that $\STP_{2}\di (\exists^{2})^{\st}$.  Fix a standard binary tree $T$.

\medskip
  
Apply overspill to $(\forall^{\st} n^{0})(\exists \beta^{0})(\beta\in T\wedge |\beta|=n)$ to obtain a sequence in $T$ of nonstandard length, say $N$.
Now define $\Phi(T)(0)$ as $0$ (resp.\ $1$) if there is a sequence $\beta^{0^{*}}\in T$ of length $N$ such that $\beta(0)=0$ (resp.\ if this is not the case).  Then define $\Phi(T)(n+1)$ as $\Phi(T)(0)*\dots \Phi(T)(n)*0$ (resp.\ $\dots*1$) if there is a sequence $\beta^{0^{*}} \in T$ of length $N$ such that $\Phi(T)(0)*\dots \Phi(T)(n)*0=\overline{\beta}(n+1)$ (resp.\ if this is not the case).  By $\STP_{2}$, there is standard $\Psi$ such that $\Phi(f)\approx_{1}\Psi(f)$ for standard $f^{1}\leq_{1}1$, and $\STP_{2}\di \UWKL^{\st}$ follows immediately.  
\end{proof}
In conclusion, the previous theorem suggests that the axiom \emph{Standard Part} is \emph{in general} fundamentally non-constructive, as claimed by Wattenberg.  
Moreover, since $\eqref{EXT}_{n+2}^{\st}$ readily follows from \emph{Transfer}, Theorem \ref{konky} is especially relevant when a proof utilises both \emph{Transfer} and \emph{Standard Part}.

\section{Compactness}\label{fraco}
We discuss Wattenberg's treatment from \cite{watje}*{III} of compactness (Sections \ref{compa} and \ref{SSTP}) and the associated extreme value theorem (Section \ref{compa2}).  

\subsection{Constructive compactness and Nonstandard Analysis}\label{compa}
Wattenberg describes the following form of compactness as `acceptable' in \cite{watje}*{III.4}.  As we will see, his choice of compactness is indeed most suitable for obtaining constructive or effective results.  
\bdefi[$F$-compactness]
A metric space $X$ is \emph{$F$-compact} if there is a standard sequence $x_{(\cdot)}$ such that $(\forall x\in X)(\forall N^{0}\in \Omega) (\exists k\leq N)(|x_{k}-x|_{X}\approx 0)$.
\edefi
Note that inside $\P$ (and extensions), the unit interval and Cantor space are $F$-compact, but not necessarily nonstandard compact (as $\P+\paai\not\vdash \STP$ by \cite{dagsam}*{\S4}).  In particular, nonstandard compactness guarantees
the infinitesimal proximity of a \emph{standard} point, while $F$-compactness states the presence of an `infinitesimal grid' of \emph{nonstandard} points.  Thus, $F$-compactness expresses the intuitive notion that a compact space `can be divided into infinitesimal pieces', a mainstay of the infinitesimal calculus used in physics and engineering.     
The notion of $F$-compactness for special cases has been studied in \cite{sambon}.  

\medskip

We first prove a basic result regarding $F$-compactness.  Note that the latter provides a kind of `discretisation' 
of the space $X$ as used in an essential way for the unit interval in the steps \eqref{koer}-\eqref{noco} at the beginning of Section \ref{good}.  
\begin{thm}[$\FC_{\R}$]
An $F$-compact $X\subset \R$ has a supremum, i.e.\ for all standard $x_{(\cdot)}$ and any $X\subset \R$, we have
\[
(\forall x\in X)(\forall N^{0}\in \Omega) (\exists k\leq N)(x_{k}\approx x) \di  (\forall x\in X)(\forall N\in \Omega)(x\lessapprox t(x_{(\cdot)}, N))),
\]
where $t(x_{(\cdot)}, N ):= \max_{i\leq N}x_{i}$.  
\end{thm}\noindent
Note that $\Omega\textsf{-CA}$ converts $t(x_{(\cdot)}, N)$ from the theorem into a \emph{standard} supremum.  

\medskip

The constructive version of the previous theorem is \cite{bish1}*{Theorem 3, p.\ 34}.  The latter version involves the notion of `totally boundedness' as in the antecedent of \eqref{cardi}, which `falls out' of the notion of $F$-compactness by the following theorem.  
\begin{thm}
From a proof $\P_{0}\vdash \FC_{\R}$, a term $s$ can be extracted such $\RCAo$ proves that for any $x_{(\cdot)}$ and $X\subset \R$ and $g^{1}$, we have
\be\label{cardi}\textstyle
(\forall k^{0},  x\in X) (\exists n\leq g(k))(|x_{n}- x|<_{\R}\frac{1}{k})\di  (\forall x\in X)(x\leq_{\R} \sup_{X}(x_{(\cdot)}))),
\ee
where the real in the consequent is defined by $\sup_{X}(x_{(\cdot)} )(k):=t(x_{(\cdot)}, s(x_{(\cdot)}, 2^{k}))$.  
\end{thm}
\begin{proof}
The nonstandard proof is trivial.  For the remaining part, one readily proves using underspill that $(\forall x\in X)(\forall N^{0}\in \Omega) (\exists k\leq N)(x_{k}\approx x)$ has the normal form
\[\textstyle
(\forall^{\st} l^{0})(\exists^{\st}M^{0})(\forall x\in X) (\exists k\leq M)(|x_{k}- x|<\frac{1}{l}).
\]
Similarly, a normal form for $(\forall x\in X)(\forall N\in \Omega)(x\lessapprox t(x_{(\cdot)}, N)))$ is as follows:
\[\textstyle
(\forall^{\st}k^{0})(\exists^{\st} M^{0})(\forall x\in X)(x\leq  t(x_{(\cdot)}, M))+\frac{1}{k}).
\]
Given these normal forms, Theorem \ref{consresultcor} now readily yields the theorem.  
\end{proof}
The function $g$ in the antecedent of \eqref{cardi} is a \emph{modulus of totally boundedness}.  

\medskip

Admittedly, the previous result is rather basic but our aim was to show that (a) $F$-compactness is converted to totally boundedness, the preferred constructive component of compactness, as in the antecedent of \eqref{cardi}, and (b) that Wattenberg correctly identifies $F$-compactness as having `constructive potential' in \cite{watje}*{III}.  

\medskip

With regard to (b), we now show that $F$-compactness cannot be (immediately) replaced with nonstandard compactness as in $\STP$ or $\STP_{\R}$.  
To this end, we first prove that $\STP$ and $\STP_{\R}$ have equivalent normal forms as noted in Section~\ref{preint}.  
Note that a version of this theorem \emph{not involving $\STP_{\R}$} may be found in \cite{dagsam}.  
\begin{thm}\label{klak}
In $\P$, $\STP$ is equivalent to $\STP_{\R}$ and to the normal forms
\begin{align}\label{frukkklk2}
(\forall^{\st}g^{2})(\exists^{\st}w^{1^{*}}, k^{0})(\forall T^{1}\leq_{1}1)\big[(\forall \alpha^{1}\in w)&(\overline{\alpha}g(\alpha)\not\in T)\\
&\di(\forall \beta\leq_{1}1)(\exists i\leq k)(\overline{\beta}i\not\in T) \big], \notag
\end{align}
\begin{align}\textstyle
(\forall^{\st} g^{2})(\exists^{\st}w^{1^{*}}, k)(\forall z\in \R)
\big[ \big(\forall y\in ( w\cap [0,1])\big)&\textstyle (|y-z|>_{\R} \frac{1}{g(y)})\label{kalkuttttt2}\\
&\textstyle\di (\forall x\in [0,1])(|x-z|>_{\R}\frac{1}{k})   \big].\notag
\end{align}
\end{thm}
\begin{proof}    
We first prove that $\STP$ and $\STP_{\R}$ are equivalent.  Now, Hirst establishes in \cite{polahirst} that $\RCA_{0}$ proves that every real $x\in [0,1]$ has a binary expansion, i.e.\ $(\forall x\in [0,1])(\exists \alpha^{1}\leq_{1}1)(x=_{\R}\sum_{i=0}^{\infty}\frac{\alpha(i)}{2^{i}})$. 
Since $\P_{0}$ proves the latter (both the internal version and the version relative to `st'), it is clear that $\STP\asa \STP_{\R}$.  

\medskip

Secondly, we prove that $\STP$ is equivalent to 
\begin{align}\label{fanns2}
(\forall T^{1}\leq_{1}1)\big[(\forall^{\st}n)(\exists \beta^{0})&(|\beta|=n \wedge \beta\in T ) \\
&\di (\exists^{\st}\alpha^{1}\leq_{1}1)(\forall^{\st}n^{0})(\overline{\alpha}n\in T)   \big].\notag
\end{align}
Assume \ref{STP} and apply overspill to $(\forall^{\st}n)(\exists \beta^{0})(|\beta|=n \wedge \beta\in T )$ to obtain $\beta_{0}^{0}\in T$ with nonstandard length $|\beta_{0}|$.  
Now apply \ref{STP} to $\beta^{1}:=\beta_{0}*00\dots$ to obtain a \emph{standard} $\alpha^{1}\leq_{1}1$ such that $\alpha\approx_{1}\beta$ and hence $(\forall^{\st}n)(\overline{\alpha}n\in T)$.  
For the reverse direction, let $f^{1}$ be a binary sequence, and define a binary tree $T_{f}$ which contains all initial segments of $f$.  
Now apply \eqref{fanns2} for $T=T_{f}$ to obtain \ref{STP}.  

\medskip

Thirdly, assume $\STP$ and note that the contraposition of \eqref{fanns2} yields:
\begin{align}\label{fannsXXY}
(\forall T^{1}\leq_{1}1)\big[ (\forall^{\st}\alpha\leq_{1}1)(\exists^{\st}n^{0})(\overline{\alpha}n&\not\in T)\di\\
 &(\exists^{\st}k^{0})(\forall \beta\leq_{1}1)(\exists i\leq k)(\overline{\beta}i\not\in T) \big].\notag
\end{align}
Since standard functionals have standard output for standard input, \eqref{fannsXXY} implies:
\begin{align}\label{fannsXXX}
(\forall T^{1}\leq_{1}1)(\forall^{\st}g^{2})\big[ (\forall^{\st}\alpha\leq_{1}1)(\overline{\alpha}g(\alpha)&\not\in T)\di\\
 &(\exists^{\st}k^{0})(\forall \beta\leq_{1}1)(\exists i\leq k)(\overline{\beta}i\not\in T) \big].\notag
\end{align}
Pushing all standard quantifiers outside as far as possible, we obtain that
\begin{align}
(\forall^{\st}g^{2})(\forall T^{1}\leq_{1}1)(\exists^{\st}k^{0}, \alpha^{1}&\leq_{1}1)\big[(\overline{\alpha}g(\alpha)\not\in T)\label{haiku}\\
&\di(\forall \beta\leq_{1}1)(\exists i\leq k)(\overline{\beta}i\not\in T) \big].\notag
\end{align}
Applying \emph{Idealisation} $\textsf{I}$, we pull the standard quantifiers to the front as follows:
\begin{align}\label{teringzooi}
(\forall^{\st}g^{2})(\exists^{\st}w^{1^{*}})(\forall T^{1}\leq_{1}1)(\exists ( \alpha^{1}\leq_{1}1,  k^{0}) \in w)\big[(\overline{\alpha}&g(\alpha)\not\in T)\\
&\di(\forall \beta\leq_{1}1)(\exists i\leq k)(\overline{\beta}i\not\in T) \big]. \notag
\end{align}
Now assume \eqref{teringzooi} and note that since $w$ is standard, all of its elements are, implying \eqref{haiku}.  
Bringing all standard quantifiers inside again (as far as possible), we obtain \eqref{fannsXXX}.  We now immediately obtain \eqref{fannsXXY} by noting that 
$(\forall^{\st}\alpha\leq_{1}1)(\exists^{\st}n^{0})(\overline{\alpha}n\not\in T)$ implies $(\exists^{\st}\Phi^{1\di 0^{*}})(\forall^{\st}\alpha\leq_{1}1)(\exists n^{0}\in \Phi(\alpha))(\overline{\alpha}n\not\in T)$ by applying $\HAC_{\INT}$ and defining $g(\alpha):=\max_{i<|\Phi(\alpha)|}\Phi(\alpha)(i)$ as in Remark \ref{simply}.   
\end{proof}
In conclusion, $\STP$ is equivalent to the normal form \eqref{frukkklk2}, and term extraction as in Theorem \ref{consresultcor} converts the latter to the special fan functional introduced in Section \ref{preint}.  Since the latter boasts extreme computational hardness, it indeed seems better to avoid nonstandard compactness $\STP$ in favour of $F$-compactness.  
In the next section, we will show how computational content can still be obtained from $\STP$ and the special fan functional.  

\subsection{The constructive status of nonstandard compactness}\label{SSTP}
We investigate \emph{Standard Part} in light of Wattenberg's claims that it be fundamentally non-constructive.  
By Theorem \ref{konky}, a rather small fragment of \emph{Standard part} is indeed fundamentally non-constructive.  
By contrast, we show in this section that $\STP$, which is the particular fragment Wattenberg uses in \cite{watje}*{II-III}, still yields constructive content, after some extra technical steps.   This section is somewhat more technical in nature as we assume familiarity with the $\ECF$-translation from  \cite{troelstra1}*{\S2.6.5}.  

\medskip

First of all, nonstandard continuity \eqref{soareyou3} clearly yields \emph{uniform} nonstandard continuity \eqref{soareyou4} for the unit interval inside $\P_{0}+\STP$ by Theorem \ref{klak}.  
Thus, the following version of $\IVT$ is immediate where the term $t$ is as in \eqref{teke}.  
\begin{thm}[$\IVT_{\textsf{\wat}}'$]\label{wakko2}
For nonstandard $N^{0}$ and nonstandard continuous $f:[0,1]\di \R$ such that $f(0)<_{\R}0<_{\R}f(1)$, $t(f,N)\in [0,1]\wedge f(t(f, N))\approx 0$.  
\end{thm}
We now show that this nonstandard version yields the following effective version.  
\begin{thm}[$\IVT_{\ef}'(s)$]
For $k^{0}$ and $f:[a,b]\di \R$ continuous with modulus $g$ and such that $f(a)<_{\R}0<_{\R}f(b)$, we have $|f(s(f,g,k))|<_{\R}\frac{1}{k}$.
\end{thm}
\begin{thm}\label{konky2}
From the proof of $\IVT_{\textsf{\wat}}'$ in $\P_{0}+\STP$, a term $u$ can be extracted such that $\RCAo$ proves $(\forall \Theta^{3})\big(\SCF(\Theta)\di \IVT_{\ef}'(u(\Theta))\big)$.
\end{thm}
\begin{proof}
Since $\P_{0}$ proves $\IVT_{\wat}$ by Theorem~\ref{konky}, $\P_{0}+\STP$ proves $\IVT_{\wat}'$, as in the latter system every nonstandard continuous function is automatically \emph{uniform} nonstandard continuous on the unit interval by $\STP_{\R}$ and Theorem \ref{klak}.  
To obtain the effective results from the theorem, one just proceeds as in Theorem \ref{hark} using the normal form \eqref{frukkklk2} of $\STP$.
\end{proof}
Secondly, Theorem \ref{konky2} is not very satisfactory as the special fan functional is not computable (in the sense of Kleene's S1-S9) in e.g.\ $(\exists^{2})$, or even the Suslin functional (See \cite{dagsam}*{\S3} for these results).  
However, the following observation (\cite{kohlenbach2}*{\S2}) by Kohlenbach will be seen to solve this problem in Theorem \ref{koreki}:
\be\label{bergje}
\text{If $\RCAo\vdash A$ then $\RCA_{0}^{2}\vdash[A]_{\ECF}$}.
\ee
Here, $\RCA_{0}^{2}$ is essentially the base theory $\RCA_{0}$ of Reverse Mathematics (\cite{simpson2}*{II}) formulated with function rather than set variables; the syntactical interpretation $[~\cdot~]_{\ECF}$ is defined in \cite{troelstra1}*{\S2.6.5} and is based on the Kleene-Kreisel model of continuous functionals.  In the latter, higher-type objects are represented by so-called \emph{associates} which is equivalent to the representation used in Reverse Mathematics by \cite{kohlenbach4}*{Prop.\ 4.4} of continuous functionals on Baire space.  

\medskip

In a nutshell, the $\ECF$-translation amounts to replacing all objects of type two or higher by \emph{type one} associates.  
Applying $\ECF$ as in \eqref{bergje} to the final part of Theorem  \ref{konky2}, we shall observe that the special fan functional is converted into a `more computable' object.  We now introduce the definition of associate for a type two functional from \cite{kohlenbach4}, and study the intuitionistic fan functional as an example.  
\bdefi[Associate]\label{LAX} The function $\alpha^{1}$ is an \emph{associate} for continuous $\Phi^{2}$ if
\begin{enumerate}
\item $(\forall \beta^{1})(\exists n^{0})(\alpha(\overline{\beta}n)>0)$,
\item $(\forall \beta^{1}, m^{0})(\alpha(\overline{\beta}m)>0 \di \Phi(\beta)+1=\alpha(\overline{\beta}m))$.
\end{enumerate}
\edefi
One often writes $\alpha(\beta)$, to be understood as $\alpha(\overline{\beta}m)-1$ for large enough $m$ as in the first item.  
Given an associate $\alpha^{1}$ for $\Phi^{2}$, an associate $\gamma^{1}$ for $\Psi^{3}$ is now defined such that $\Psi(\Phi)$ is $\gamma(\alpha)$ where the latter is again $\gamma(\overline{\alpha}k)-1$ for large enough $k$.  

\medskip

By way of an example, consider the intuitionistic fan functional as in $\MUC(\Omega)$ from Section \ref{firstkol}. 
Following the heuristic that all objects of type two or higher are replaced by associates by $\ECF$, it is straightforward to see that $[\MUC(\Omega)]_{\ECF}$ is:
\begin{align}
(\forall \alpha^{1})\Big[ &\underline{(\forall f^{1})(\exists n^{0})(\alpha(\overline{f}n)>0)}\di\big( (\exists m^{0})(\gamma(\overline{\alpha}m)>0)  ~ \wedge \tag{$\TOF(\gamma^{1})$}     \\
& (\forall g^{1}, h^{1}\leq_{1}1, k^{0})([\gamma(\overline{\alpha}k)>0\wedge \overline{h}\gamma(\overline{\alpha}k)=_{0} \overline{g}\gamma(\overline{\alpha}k)]  \di \alpha(h)=\alpha(g) >0 ) \big)\Big].\notag
\end{align}
The underlined formula expresses that $\alpha^{1}$ is an associate representing a (continuous) functional $Y^{2}$, while $\TOF(\gamma)$ expresses that $\gamma^{1}$ is an associate for the intuitionistic fan functional, i.e.\ $\Omega(Y)$ as in $\MUC(\Omega)$ is given by $\gamma(\alpha)$, and the rest of $\TOF(\gamma)$ ensures that $\gamma(\alpha)$ makes sense.  

\medskip

Note that if a functional $\Phi^{2}$ has an associate as in Definition \ref{LAX}, it is \emph{automatically} continuous on Baire space.  
Thus, since $\mu^{2}$ as in $(\mu^{2})$ is discontinuous (e.g. at $00\dots$), $[(\mu^{2})]_{\ECF}$ is equivalent to $0=1$.  
In particular, we observe that the $\ECF$-translation replaces \emph{any} type two variable with a type one variable over associates, i.e.\ the new variable 
ranges over (representations of) \emph{continuous} functionals.  

\medskip

As a result of the aforementioned `continuous replacement', $\WKL$ is equivalent $(\exists \gamma^{1})\TOF(\gamma)$, and the latter is of course $[(\exists \Omega^{3})\MUC(\Omega)]_{\ECF}$.  As it turns out, the intuitionistic fan functional even has a \emph{primitive recursive} associate $\widehat{\FAN}$ which may be found in \cite{noortje}*{p.\ 102}. We thus also have $\WKL\asa \TOF(\widehat{\FAN})$ (See \cite{longmann}*{\S7.3.4}).  
The following theorem shows that the special fan functional becomes `more computable' thanks to $\ECF$.  We shall make use of the nonstandard axiom
\be\tag{$\NUC$}
(\forall^{\st}Y^{2})(\forall f^{1}, g^{1}\leq_{1}1)(f\approx_{1}g\di Y(f)=_{0}Y(g)), 
\ee
Note that $\NUC$ expresses that every type two functional is nonstandard uniformly continuous on Cantor space, akin to Brouwer's continuity theorem (\cite{brouw}).
\begin{thm}\label{koreki} From $\P_{0}\vdash {\NUC}\di \STP$, terms $t_{0}, t_{1}$ can be extracted such that $\RCAo\vdash (\forall \Omega)(\MUC(\Omega)\di \SCF(t_{0}(\Omega)))$ and $\RCA_{0}^{2}+\WKL\vdash [\SCF]_{\ECF}(t_{1}(\widehat{\FAN}))$.
\end{thm}
\begin{proof}
As in Theorem \ref{hark}, the normal form of $\NUC$ is readily obtained as follows:  
\be\label{XXX}
(\forall^{\st}Y^{2})(\exists^{\st}N)(\forall f^{1}, g^{1}\leq_{1}1)(\overline{f}N=\overline{g}N\di Y(f)=_{0}Y(g)), 
\ee
Applying $\HAC_{\INT}$ to \eqref{XXX} as in Remark \ref{simply} yields that    
\[
(\exists^{\st} \Omega^{3})(\forall^{\st}Y^{2})(\exists^{\st}N)(\forall f^{1}, g^{1}\leq_{1}1)(\overline{f}\Omega(Y)=\overline{g}\Omega(Y)\di Y(f)=_{0}Y(g)).
\]
To obtain the nonstandard implication, note that $\STP$ is equivalent to \eqref{frukkklk2} and define $ k^{0}$ as in the latter as the (clearly standard) maximum of $g(\sigma*00)$ for all binary $\sigma^{0^{*}}$ of length $\Omega(g)+1$, while $w^{1}$ is the (standard) collection of all $\sigma_{i}*00\dots$ where the binary $\sigma_{i}^{0^{*}}$ has length $\Omega(g)+1$.  Hence, $\NUC$ implies \eqref{frukkklk2} and hence $\STP$.  
Since $\NUC$ and $\STP$ have normal forms \eqref{XXX} and \eqref{frukkklk2}, applying term extraction as in Theorem \ref{consresultcor} to $\P_{0}\vdash [\NUC\di\STP]$ readily yields the term $t_{0}$ from the theorem.  
  
\medskip

Finally, applying the $\ECF$-translation as in \eqref{bergje} to $\RCAo\vdash (\forall \Omega)(\MUC(\Omega)\di \SCF(t_{0}(\Omega)))$, we obtain 
that $\RCA_{0}^{2}\vdash [ (\forall \Omega)(\MUC(\Omega)\di \SCF(t_{0}(\Omega)))]_{\ECF} $, which becomes 
$\RCA_{0}^{2}\vdash (\forall \gamma^{1})(\TOF(\gamma)\di [\SCF]_{\ECF}(t_{1}(\gamma)))$ as $[\MUC(\Omega)]_{\ECF}$ is $\TOF(\gamma)$.
Since $\WKL\asa \TOF(\widehat{\FAN})$ by \cite{longmann}*{\S7.3.4}, the theorem now follows.  
\end{proof}
Note that parts of Theorem \ref{koreki} may be found in \cite{samGH, dagsam}.  We have the following corollary pertaining to Theorem \ref{konky2}.
\begin{cor}\label{klike}
From the proof of $\IVT_{\textsf{\wat}}'$ in $\P_{0}+\STP$, a term $v^{1}$ can be extracted such that $\RCA_{0}^{2}+\WKL$ proves 
$[\IVT_{\ef}']_{\ECF}(v(\widehat{\FAN}))\big)$.
\end{cor}
\begin{proof}
Apply the $\ECF$-translation to the conclusion of Theorem \ref{konky2}.
\end{proof}
Note that $\WKL$ is non-constructive, but the term $v(\widehat{\FAN})$ is computable.  
Furthermore, the only real modification the $\ECF$-translation bestows upong the intermediate value theorem from Theorem \ref{konky2} is the replacement of continuous functions by associates (which can always be done given $\WKL$ by \cite{kohlenbach4}*{Theorem 4.6}).  

\medskip

In conclusion, we have observed that $\STP$ is indeed non-constructive in nature in that it gives rise to the special fan functional 
as in Theorem \ref{konky2}.  However, a somewhat technical detour (using the $\ECF$-interpretation) still yields computational information as in Corollary \ref{klike}.


\subsection{Extreme value theorem}\label{compa2}
We briefly discuss Wattenberg's treatment from \cite{watje}*{III} of Weierstra\ss' extreme value theorem inside Nonstandard Analysis.  
\subsubsection{Preliminaries}
The extreme value theorem $(\WMX)$ is a basic result from calculus and can be formulated as follows.  
\begin{thm}[$\WMX$]
Suppose that $X$ is compact and that $f:X\di \R$ is continuous. Then there is $x\in X$ such that $(\forall y\in X)(f(y)\leq_{\R}f(x))$.  
\end{thm}
As is well-known, $\WMX$ implies a non-trivial fragment of the law of excluded middle (See e.g.\ \cite{beeson1}*{I.6} or \cite{mandje}) and is therefore rejected in constructive mathematics.
A slight modification of Wattenberg's nonstandard version of $\WMX$ will be shown to yield the following \emph{effective} version in Section~\ref{goof}, similar to \cite{bish1}*{p.\ 89}.  
\begin{thm}[$\WMX_{\ef}(s)$]
For $k^{0}$ and $f:X\di \R$ uniformly continuous with modulus $g$ on the compact space $X$ with modulus of totally boundedness $h$, we have $(\forall x\in X)( f(x)\leq_{\R}f(s(f,g,h,k))+\frac{1}{k})$.
\end{thm}
Note that this version no longer involves Nonstandard Analysis.  Furthermore, the term $s$ is `read off' from the (modified) Wattenberg proof.  
Thus, Wattenberg's claims about the effective content of Nonstandard Analysis are again at least partially correct.   
One the other hand, Wattenberg explicitly uses \emph{Transfer} in the proof of \cite{watje}*{III.6}, which is problematic if one is interested in computational content, as was established above.

%

\subsubsection{Constructive extreme value theorem and Nonstandard Analysis}\label{goof}
Wattenberg proves various nonstandard versions of $\WMX$ inside Nonstandard Analysis in \cite{watje}*{III.7}.  
He refers to (a trivial reformulation of)  the Theorem \ref{wakko4} below as 
\begin{center}
\emph{a completely ``constructive'' version of the Extreme Value Theorem}
\end{center}
in \cite{watje}*{p.\ 308}.  Note that Wattenberg uses the notion of an `implementation' (See \cite{watje}*{Def.\ III.8}) rather than nonstandard uniform continuity, although both essentially amount to the same thing in this context.  
\begin{thm}[$\WMX_{\textsf{\wat}}$]\label{wakko4}
For nonstandard $N^{0}$, $F$-compact $X$ with standard sequence $x_{(\cdot)}$, and nonstandard uniformly continuous $f:X\di \R$, we have $(\forall x\in X)(f(x)\lessapprox f(t(x_{(\cdot)}, N)))$.  
\end{thm}\noindent
The term $t^{((1\di 1)\times 0)\di 0}$ from the theorem is defined as follows (where $h:=\frac{b-a}{N}$):
\be\label{teke}
t(f,N):=
\begin{cases}
(\mu j \leq N)\big(\big[f\big(jh\big)\big](2^{N})\leq_{0} 0\big)h & \textup{if such exists}\\
N+1 & \textup{otherwise}
\end{cases}.
\ee
Wattenberg proves $\WMX_{\wat}$ in \cite{watje}*{III.7} using (what amounts to) the following:
\begin{enumerate}  
\renewcommand{\theenumi}{\roman{enumi}}
\item Fix a standard sequence $x_{(\cdot)}$ and nonstandard $N^{0}$ as provided by the $F$-compactness of $X$. \label{koer2}
\item Define $t_{0}:=x_{0}$ and $t_{n+1}:= t_{n}$ if $f(x_{n+1})\lessapprox f(t_{n})$ and $t_{n+1}=x_{n+1}$ if $f(x_{n+1})\gtrapprox f(t_{n})$.  Note that for $n\leq N$, we have $f(x_{n})\lessapprox f(t_{N})$. \label{fra2g}
\item Since for every $x\in X$ there is $j\leq N$ such that $x_{j}\approx x$, we have $(\forall x\in X)(f(x)\lessapprox f(t_{N}))$ by continuity.  \label{cruxjes2}
\end{enumerate}
Note that item \eqref{cruxjes2} makes use of \emph{uniform} nonstandard continuity.  When working in $\IST$, one would apply $\STP$ to $t_{N}$ to obtain a \emph{standard} maximum for $f$.  Note that a similar remark regarding $\Omega\textsf{-CA}$ as in Remark \ref{notnotnting} applies to $\WMX_{\wat}$.

\medskip

Finally, we have the following theorem.  
\begin{thm}\label{hark4}
From the proof of $\WMX_{\textsf{\wat}}$ in $\P_{0}$, a term $u$ can be extracted such that $\RCAo$ proves $\WMX_{\ef}(u)$.
\end{thm}
\begin{proof}
The proof of $\WMX_{\wat}$ inside $\P_{0}$ follows from the above steps (i)-(iii), assuming we use approximations (say up to precision $2^{N}$ for $N$ from item (i)) to $f(x_{i})$ and $f(t_{i})$.  
One then applies Theorem \ref{consresultcor} to $\P_{0}\vdash \WMX_{\wat}$ to obtain the term $u$ from the theorem, following Template \ref{tempe}.  
\end{proof}

\bye